%% file: convergence_analysis_latest.tex
\documentclass[11pt]{article}
  \pdfoutput=1

  \input{slidesymbols}

  \usepackage{stmaryrd}
  \usepackage{amsmath,amsthm,amssymb}
  \usepackage{graphicx,color}
  \usepackage{float}
  \usepackage{dsfont}
  \usepackage{float}
  \usepackage[]{algorithm2e}

  \usepackage{caption}
  \usepackage{subcaption}

  \usepackage{makecell}
  \usepackage{pgfplots}
  \usepackage{fancyhdr}
  \usepackage{relsize}
  
  \usepackage{accents} 

  \usepackage{hyperref}

 \providecommand{\keywords}[1]{\textbf{\textit{Keywords: }} #1}

  \numberwithin{equation}{section}
  \usepackage{tikz-cd}
  
  \usepackage{accents}

  \usepackage{tikz}
  \usetikzlibrary{patterns}
  \tikzset{point/.style={insert path={ node[scale=2.5*sqrt(\pgflinewidth)]{.} }}}
  \usetikzlibrary{calc,arrows}
  
  \def\centerarc[#1](#2)(#3:#4:#5){ \draw[#1] ($(#2)+({#5*cos(#3)},{#5*sin(#3)})$) arc (#3:#4:#5); }

 \newcommand{\ac}{\accentset{\circ}}

 \newcommand{\Hrep}{\mathcal H(\mathcal X,\VR^d)}
\newcommand{\VXM}{\mathcal V^{M}_{\mathcal X}(\VR^d,\VR^d)}
\newcommand{\NXM}{\mathcal N^{M}_{\mathcal X}(M,\VR^d)}

\newcommand{\VMM}{\mathcal V^{M}(\VR^d,\VR^d)}

\newcommand{\VOmega}{\mathcal V^{\partial\Omega}(\VR^d,\VR^d)}

  \newtheorem{thm}{Theorem}[section]
  \theoremstyle{plain}

  \newtheorem{lemma}[thm]{{\textbf Lemma}}
  \newtheorem{theorem}[thm]{{\textbf Theorem}}
  
  \newtheorem{corollary}[thm]{{\textbf Corollary}}
  \newtheorem{assumption}[thm]{{\textbf Assumption}}
  
  \newtheorem{definition}[thm]{{\textbf Definition}}
  \newtheorem{remark}[thm]{{\textbf Remark}}
  
  \newtheorem{example}[thm]{{\textbf Example}}

  \newcommand{\overbar}[1]{\mkern 1.5mu\overline{\mkern-1.5mu#1\mkern-1.5mu}\mkern 1.5mu}

  \newcommand{\eps}{\epsilon}

 \newcommand{\llb}{\llbracket}
 \newcommand{\rrb}{\rrbracket}

\newcommand{\ben}{\begin{equation}}
\newcommand{\een}{\end{equation}}

\newcommand{\beq}{\begin{eqnarray}}
\newcommand{\eeq}{\end{eqnarray}}

\newcommand{\benn}{\begin{equation*}}
\newcommand{\eenn}{\end{equation*}}

  \title{Convergence of Newton's method in shape optimisation via approximate normal functions}

  \author{Kevin Sturm\thanks{Universit\"at Duisburg-Essen, Fakult\"at f\"ur Mathematik, Thea-Leymann-Str. 9, D-45127 Essen, Germany (kevin.sturm@uni-due.de)} }
  \date{}
 
\begin{document}
  	
  	\maketitle

\begin{abstract}
In this paper we propose a Newton 
method for shape functions defined on an image set generated by the (Micheletti) metric group. We review basic 
properties of the metric group and a quotient associated with the metric group and a fixed domain. 

Taking into account the special structure of the second shape derivative and its 
symmetric part allows us to distinguish between two Hessians, the \textit{domain shape Hessian} and the \textit{boundary shape Hessian}.

Using the domain Hessian we define a Newton method on the metric group by discretising the tangent space of the quotient via approximate normal functions using reproducing kernels. Under suitable 
assumptions we are able to show superlinear convergences of the Newton iterations 
and additionally convergence of the shapes in the 
metric group. Finally we verify our findings in a number of numerical experiments including a thorough numerical study of the impact of the discretisation on the convergence speed.  
\end{abstract}
  	

\section*{Introduction}
Shape optimisation is concerned with the minimisation of real-valued shape functions $J(\Omega)$ over an admissible set $\mathcal A$ containing a collection of subsets $\Omega\subset \VR^d$; see \cite{HenPie05,SokZol89,DelZol11,MR1969772}.
Many tasks and processes in industry can be optimised using shape optimisation methods. Therefore it is
of paramount importance to find efficient methods to solve these problems numerically. 

The aim of this paper is to develop a Newton 
algorithm to find stationary points of shape functions defined on an image set generated by the metric group $\Cf := \Cf(C^1)$; cf \cite[Chapter 
3]{DelZol11} and \cite{MR0377306,ubt_epub202}. For every fixed set $\omega\subset \VR^d$ the image set $\mathcal A_\omega :=\mathcal Z(\omega)$ consists of all images $F(\omega)$, where $F:\VR^d\rightarrow \VR^d$ 
belongs to the metric group $\Cf$. This image 
set can be identified with the quotient $\Cf/\Cg_\omega$ that identifies transformations in $\Cf$ with the same image on $\omega$. Using the special 
structure of the second Euler derivative and its symmetric part allows us to define two shape Hessians, the \textit{domain (shape) Hessian} 
and the \textit{boundary (shape) Hessian}. The domain and boundary Hessian are functions defined on the tangent space of the metric group and its quotient space, respectively, and coincide when they are restricted to normal perturbations on the boundary. We also establish a new 
proof of the structure theorem for the symmetric part of the second Euler derivative;  \cite{MR1930612,MR1625559}.

In order to approximate the Newton equation we need to approximate the tangent space or a subspace of  $\Cf/\Cg_\omega$  in a suitable way. For this purpose we introduce for every $C^1$ submanifold $M$ of co-dimension one (having in mind $M=\partial \Omega$) so called \textit{approximate normal functions}.  The properties of the reproducing kernel ensure that these functions are linearly independent. Additionally approximate normal functions are  approximately normal along $M$ and thus are suitable functions to approximate
a subset of the tangent space of the quotient space $\Cf/\Cg_\omega$. In order to have a sparse Hessian approximation we work with compactly supported reproducing kernels. A key ingredient of the proof is a transport that relates approximate normal functions on different domains.   This allows us to show superlinear convergence of Newton's method in the discrete setting. 
Our analysis reveals that quadratic convergence cannot be expected when normal fields
are approximated.

Second order methods such as Newton and Newton-like methods have the great advantage over gradient methods that they converge superlinearly or 
even quadratically. 
Despite their importance, the literature on second order methods for shape optimisation problems
is incomplete and only a limited number of papers use second order information; see \cite{MR1777697,MR2211069,MR2049659,MR2049779,MR2192593,MR2144354,Allaire2016,MR1759037,MR3518366}. Convergence 
analysis of second order methods is even less studied; \cite{MR3201954,MR2049779,MR2299620}.

One reason for the lack of literature in this field is the notorious nonlinearity of the space of admissible shapes which leads to nonconvex optimisation problems. However in some situations it is possible to turn admissible sets into a (mostly Riemannian) manifold and therefore tools from differential geometry become accessible.
Newton methods, Newton-like and gradient methods on finite dimensional Riemannian manifolds were already subject of intensive research
\cite{MR2364186,MR2968868}. In shape optimisation the spaces of shapes are at best infinite dimensional manifolds and in this situation the analysis is more 
complicated as one has to account for the infinite dimensionality of 
the manifold;\;\cite{MR1471480,MR583436}.   In the recent work \cite{MR3201954} the 
link between shape optimisation problems and a certain infinite 
dimensional Riemannian manifolds of mappings, also called shape space, has been established.
To be more specific the analysis 
was carried out in the so-called shape space of plane curves studied in \cite{MR2201275}. 
In this paper we want to provide another approach employing the Micheletti metric space. 

\paragraph*{Structure of the paper}
\mbox{}\\
In Section 1 we recall the definition of the 
 metric group $\Cf$ and its basic properties. 

In Section 2, we recall the structure of first and second shape derivatives.  
We give a new proof of the structure of the symmetric part of the second 
derivative (referred to as third structure theorem). 
Then  we introduce two shape Hessians, the domain shape Hessian and boundary shape Hessian defined on the tangent space of $\Cf$ and  $\Cf/\Cg_\omega$, respectively. 

In Section 3, we use reproducing kernels to introduce novel approximate normal basis functions. These functions yield an approximation of subspace of the 
tangent space of the quotient $\Cf/\Cg_\omega$. It turns out that the domain and boundary shape Hessian restricted to 
the space of approximate normal functions are approximately the same. As a result as long as we are close
to a stationary point we can use the domain Hessian instead of the boundary Hessian.

In Section 4, we introduce and study a Newton method using the approximate normal functions from Section 3. 
A careful analysis shows that, under suitable conditions, the Newton method converges superlinear. The 
generated transformations which correspond to the shapes convergence in the metric
of $\Cf$.

Section 6 provides some numerical results comparing a gradient method with Newton's methods defined by different Hessians. We show experiments employing the domain, boundary and Riemannian shape Hessian \cite{MR3201954}. These 
results are compared with a standard gradient algorithm and show the superiority of Newton's method near 
a stationary point.

\section{Micheletti's metric group and its properties}
 This section builds the basis upon which we will develop our Newton method and its convergence proof. Particularly we introduce function spaces, define the Micheletti metric group, and recall some of its properties. 
\subsection{Function spaces}
Throughout this paper $\Dsf \subset \VR^d$ is an open set. 
We denote the space of continuous vector fields on $\overbar\Dsf$ vanishing on $\partial \Dsf$ by 
\[
\ac C(\overbar \Dsf,\VR^d) = \{f:\overbar\Dsf \rightarrow \VR^d: \; f \text{ is continuous and } f=0\text{ on } \partial \Dsf\}.
\]
We denote by  $C^k(\Dsf,\VR^d)$, $k\ge 1$, the usual space of $k$-times continuously 
differentiable functions on $\Dsf$ with values in $\VR^d$.
The space $C^k(\overbar\Dsf,\VR^d)$ comprises all functions from $C^k(\Dsf,\VR^d)$ that admit a uniformly continuous and bounded extensions of its partial derivatives $\partial_\alpha f$ to $\overbar{\Dsf}$ for all  multi-indices $\alpha=(\alpha_1,\ldots, \alpha_d)\in \VN^d$  satisfying  $|\alpha|\le k$. 
The space $C^k_b(\Dsf ,\VR^d)$ indicates all $k$-times differentiable 
functions $f$ on $ \Dsf$ with values in $\VR^d$ that have bounded and continuous partial derivatives $\partial_\alpha f$ 
for all multi-indices $|\alpha|\le k$.  We equip the spaces $C^1(\overbar \Dsf,\VR^d)$ and $C^1_b( \Dsf,\VR^d)$ with the norm
$
\|f\|_{C^1} := \sup_{x\in\overbar \Dsf} \|f(x)\| + \sup_{x\in\overbar \Dsf}\|\partial f(x)\|, 
$
where $\partial f$ denotes the first derivative of $f$.  

For all spaces introduced above we define subspaces:
 $\ac C^k(\overbar{\Dsf},\VR^d) :=C^k(\overbar{\Dsf},\VR^d)\cap \ac C(\overbar\Dsf,\VR^d)$ and 
$\ac C^k_b(\Dsf,\VR^d) :=C^k_b(\Dsf,\VR^d)\cap \ac C(\overbar\Dsf,\VR^d)$.   It is worth nothing that $C^k(\overbar \VR^d,\VR^d)\ne C^k(\VR^d,\VR^d)$.

 The flow $\Phi_t^X=\Phi_t$ of a vector field $X\in \ac C^1(\overbar \Dsf,\VR^d)$ is defined by $\Phi_t^X(x_0):=x(t,X)$  for $x_0\in \overbar\Dsf$ and $t\ge 0$, where $x(\cdot,X) = x(\cdot)$ is the solution of $x'(t) = X(x(t))$, $t\ge 0$ and $x(0)=x_0$; see \cite[pp. 131]{MR2413709}. 

\subsection{Group of transformations and metric}
We begin with the definition of the metric group $\Cf$ and review 
some of its basic properties; see \cite[Chapter 3]{DelZol11}. 
\begin{definition}[{\cite[p.124]{DelZol11}}]
	The Micheletti group associated with the Banach space $C^1(\overbar\VR^d,\VR^d)$ is defined by 
	\ben
	\mathcal{F}  := \{\Id+f:\VR^d\rightarrow\VR^d \text{ bijective}: \; f\in C^1(\overbar\VR^d,\VR^d), \; \exists g\in 
	C^1(\overbar\VR^d,\VR^d) \text{ so that } (\Id + f)^{-1} = \Id + g\}.
	\een
\end{definition}
This set is a group under composition $(F_1\circ F_2)(x):= F_1(F_2(x))$. The transformations $F=\Id+f$ in $\Cf$ are unbounded, since the identity mapping $\Id$ on $\VR^d$ is unbounded and $f$ is bounded. However, their derivative $\partial F = I+\partial f$ is bounded since the identity matrix $I\in \VR^{d,d}$ and  $\partial f$ are both bounded on $\VR^d$. 
\begin{definition}[{\cite[p.126]{DelZol11}}]
	The distance between the identity mapping $\Id$ on $\VR^d$ and $F\in \Cf$ is defined by 
	\ben\label{eq:micheletti}
	d(\text{id},F) := \inf_{\substack{ F=(\Id+f_1)\circ\cdots (\Id+f_n),\\ 
			n\in \VN,\; \Id+f_i\in \Cf } } \sum_{k=1}^n 
	\|f_k\|_{C^1} +  \|f_k\circ (\Id+f_k)^{-1}\|_{C^1}.
	\een
	The distance between arbitrary $F_1,F_2\in \Cf$ is 
	defined by 
	$d(F_1,F_2):= d(\Id, F_2\circ F_1^{-1}).$
\end{definition}
It is readily checked that $d(\cdot, \cdot)$ is right-invariant, that is, 
$d(F_1\circ G, F_2\circ G)=d(F_1,F_2)$ for all $F_1,F_2,G\in \Cf$.  The symmetry follows from the right-invariance and the definition of the metric. 
For a proof that $d$ satisfies the triangle inequality and the completeness of  $(\Cf
, d(\cdot,\cdot))$ we refer to \cite[p.134, Theorem 2.6]{DelZol11}.

\subsection{Image sets and subgroup}
In shape optimisation the metric space $(\Cf,d)$ is used as follows.  We 
take an arbitrary set $\omega\subset \VR^d$ and associate with it the image set 
\ben\label{eq:image_set_X}
 \mathcal Z(\omega)  := \{(\Id+f)(\omega):\; \Id+f\in \Cf \}.
\een
This set forms the set of all admissible shapes on which a shape function $J(\cdot)$ is to be minimised. In the 
following sections we study a Newton method that aims to find stationary points of a shape function $J:\mathcal Z(\omega)\rightarrow \VR$.

A set $\Omega \in \mathcal Z(\omega)$ does not correspond to a unique $F\in \Cf$ as two elements $F,\tilde F\in \Cf$ can have the same image $F(\omega)=\tilde F(\omega)$.
Therefore we identify transformations whose image coincides on $\omega$. For this purpose we define a subgroup of $\Cf$ by  
\ben\label{eq:subgroup_G}
\Cg_\omega := \{F\in \Cf:\; F(\omega)=\omega \}.
\een
It is readily checked that $\Cg_\omega$ is a subgroup of $\Cf$ and hence the quotient  $\Cf/\Cg_\omega$ is well-defined. It can also be shown that $\Cf/\Cg_\omega$ equipped with the quotient metric is a complete metric space (\cite[Theorem 2.8, p. 141]{DelZol11}) itself if for example 
$\Omega$ is a smooth domain or open and crack free ($\text{int}(\overbar{\Omega})=\Omega$); see \cite[Chapter 3]{DelZol11}. Henceforth we denote 
the equivalence classes of $\Cf/\Cg_\omega$ by $[F]$. 
\begin{definition}\label{def:identification}
The set  $\mathcal Z(\omega)$  and the quotient $\Cf/\Cg_\omega$ are identified via the bijection
$ j_{\omega}: \Cf/\Cg_\omega \rightarrow  \mathcal Z(\omega) $ that maps the equivalence class $[F]$ to its images $F(\omega)$. Every function $f:\Cf/\Cg_\omega\rightarrow \VR$ is identified with $\tilde f: \mathcal Z(\omega)  \rightarrow \VR$ via $\tilde f := f\circ j^{-1}_\omega.$
\end{definition}

\subsection{Properties of the metric}
Let us now extract some refined properties of the metric $d(\cdot,\cdot)$. These properties are used later for the proof of our Newton method. 
 We show that if the norm of $f\in C^1(\overbar\VR^d,\VR^d)$ is smaller than one, then 
the distance $d(\Id,\Id+f)$ can be estimated from above. 
\begin{lemma}\label{lem:estimate_metric}
	Let $q\in (0,1)$ be arbitrary.	For all $\Id+f\in \Cf$ such that $\|f\|_{C^1}<q$, we have
	\ben
	d(\Id,\Id+f) \le \|f\|_{C^1} + \|f\|_\infty +  1/(1-q)  \|\partial f\|_\infty (\|\partial f\|_\infty + 1). 
	\een
	Particularly $d(\Id,\Id+f) \le p_2(\|f\|_{C^1})$ with $p_2(r) := (2+1/(1-q))r + (1/(1-q))r^2$.
\end{lemma}
\begin{proof}
	Firstly by definition of 
	$d(\cdot,\cdot)$ as an infimum,
	$d(\Id,\Id + f) \le \|f\|_{C^1} + \|f\circ (\Id+f)^{-1}\|_{C^1}$
	for all $\Id+f\in \Cf $. As $\Id+f$ is a bijection, we have $\|f\circ (\Id+f)^{-1}\|_\infty = \|f\|_\infty$. 
	By the chain rule we obtain $\partial (f\circ (\Id+f)^{-1}) = (\partial f(I+\partial f)^{-1})\circ (\Id+f)^{-1}$ and thus using again that $\Id+f$ is a bijection gives
	\ben\label{eq:estimate_partial_f}
	\begin{split}
		\|\partial (f\circ (\Id+f)^{-1})\|_\infty \le \|\partial f\|_\infty\| 
		(I+\partial f)^{-1}\|_\infty.
	\end{split}
	\een
	Let $\text{inv}(A):=A^{-1}$ denote the inverse mapping defined for all 
	invertible $A\in \VR^{d,d}$.  
	For given  invertible  $A_0\in \VR^{d,d}$ and $A\in 
	\VR^{d,d}$ with $\|A-A_0\|<q/\|A_0^{-1}\|$, we get by \cite[Satz 7.2]{AmmEschII}
	the Lipschitz estimate $ \|\text{inv}(A)-\text{inv}(A_0)\| < 1/(1-q)\|A_0^{-1}\|^2\|A-A_0\| $. 
	It follows by the triangle inequality
	$
	\|\text{inv}(A)\| < 1/(1-q)\|A_0^{-1}\|^2\|A-A_0\| + \|\text{inv}(A_0)\|. 
	$
	Hence setting $A_0:=I$ and $A:=I+\partial f(x)$ for fixed $x\in \VR^d$ yields
	$
	\|(I+\partial f)^{-1}(x)\| < 1/(1-q)\|\partial f(x)\| + 1. 
	$ 
	Thus using this estimate in \eqref{eq:estimate_partial_f} we arrive at
	$
	\|\partial (f\circ (\Id+f)^{-1})\|_\infty \le   1/(1-q) \|\partial f\|_\infty (\|\partial f\|_\infty + 1)$
	and this finishes the proof. 
\end{proof}

The next lemma shows a statement similar to Lemma~\ref{lem:estimate_metric}, but without 
the assumption that the norms of $f_i$ being smaller than one.  However, the estimate
is not as sharp. We also refer to \cite[p. 127, Example 2.2]{DelZol11} where 
the Banach space of bounded Lipschitz continuous functions rather than $C^1(\overbar\VR^d,\VR^d)$ is considered.  
\begin{lemma}\label{lem:continuity_group}
	For all $(\Id+f_k)_{k=1,\ldots, n}$ in $\Cf$, $n\ge 0$,  we have
	\ben\label{eq:estimate_f_k}
	\|(\Id+f_1)\circ \cdots \circ (\Id+f_n)-\Id\|_{C^1} \le 
	 e^{\left(\sum_{k=1}^n \|\partial f_k\|_\infty\right)} \sum_{k=1}^n \|\partial f_k\|_\infty.
	\een
\end{lemma}
\begin{proof}
	The proof follows the lines of \cite[p. 127, Example 2.2]{DelZol11} and is therefore deferred to the appendix. 
\end{proof}

With the help of the previous lemma we can show that the convergence of $(F_n)$  to $F$ 
in $\Cf$ implies the convergence of $F_n-\Id$ and
$F_n^{-1}-\Id$ to  $F-\Id$ and $F^{-1}-\Id$ in $C^1(\overbar\VR^d,\VR^d)$, respectively. This statement is summarised in the following lemma.
\begin{lemma}\label{lem:convergence_F_n}
	Let $F_n,F\in \Cf$ be given and assume $F_n\rightarrow F
	$ in $\Cf$ as $n\rightarrow \infty$. Then
	\ben
	F_n-\Id \rightarrow F-\Id \quad \text{ and }\quad F_n^{-1}-\Id \rightarrow F^{-1}-\Id \quad \text{ in }C^1(\overbar\VR^d,\VR^d) \quad \text{ as }n\rightarrow \infty.
	\een
\end{lemma}
\begin{proof}
	Thanks to the right invariance of metric $d$ we have 
	$d(F_n,F)=d(\Id,F\circ F_n^{-1})=d(\Id,F_n\circ F^{-1})$.  Therefore 
	we may assume without loss of generality that $F=\Id$ and $F_n\rightarrow \Id$ and 
	$F_n^{-1}\rightarrow \Id$ in $\Cf$.  
	By assumption for every $\eps >0$ we find $N\ge 1$ such that  $d(F_n,\Id)<\eps$ 
	for all $n\ge N$. By definition of $d(\cdot,\cdot)
	$ as an infimum we find for every number $n\ge N$, a number $M\ge 1$ and transformations  $(\Id+f_i^n)_{i=1}^M\in \Cf$ such that $F_n = (\Id+f_1^n)\circ \cdots \circ (\Id+f_M^n)$ and
	\ben
	d(F_n,\Id) \le \sum_{k=1}^M 
	\|f_k^n\|_{C^1} +  \|f_k^n\circ (\Id+f_k^n)^{-1}\|_{C^1}<\eps. 
	\een
	Now Lemma~\ref{lem:continuity_group} yields
	$
	\|F_n-\Id\|_{C^1} \le e^{\eps}\eps
	$
	for all $n\ge N$. Since $\eps$ was arbitrary we conclude 
	$F_n-\Id \rightarrow 0$ as $n\rightarrow \infty$. Noticing 
	$d(\Id,F_n^{-1}) = d(\Id,F_n)\rightarrow 0$ as $n\rightarrow \infty$ shows that the argumentation above 
	can be repeated to prove 
	$F^{-1}_n-\Id \rightarrow 0$ as $n\rightarrow \infty$ which finishes the proof. 	
\end{proof}

\subsection{Parametrisations of $\Cf$}
 In the following lemma $B_\delta(0)$ denotes the open ball 
in $C^1(\overbar \VR^d,\VR^d)$ with radius $\delta>0$ centered at the origin.
\begin{lemma}\label{lem:micheletti_manifold}
Let $q \in (0,1)$ be arbitrary.	For each 
	$F\in \Cf$ the mapping 
	\ben\label{eq:parametrisation}
	\psi_F:B_{\delta_F}(0) \rightarrow \Cf:\; g\mapsto F+g, 
	\een
	 $\delta_F:=\min\{1/\|\partial F^{-1}\|_{\infty},q\}$, is a well-defined parameterisation of a neighborhood of $F$. 
	Differentiable charts are given by $\varphi_F(H) := \psi_F^{-1}(H) = F-H$ with 
	$U_F := \psi_F(B_{\delta_F}(0))$. Additionally, the sets $U_F$ are open in $(\Cf,d)$. 
\end{lemma}
\begin{proof}
	We first show that for given $F\in \Cf$ the mapping
	\ben
	\psi_F:B_\delta(0) \rightarrow \Cf:\; g\mapsto F+g 
	\een
	is well-defined when we choose $\delta_F := \min\{1/\|\partial F^{-1}\|_{\infty},q\}$. Indeed we can write
	$F+g = (\Id +g \circ F^{-1})\circ F$.  By the choice of $\delta_F$ we have 
	$ \|g \circ F^{-1}\|_{C^1}<1$ and hence \cite[Theorem 2.14, (i), p.148]{DelZol11} implies that the chart is well-defined.

	Next we show that the chart change is smooth. Let 
	$F_1,F_2\in \Cf$ be given.  The chart change is given by
	\ben
	\varphi_{F_1}\circ \varphi_{F_2}^{-1}: \varphi_{F_2}(B_{\delta_{F_1}}(0)\cap B_{\delta_{F_2}}(0)) \rightarrow 
	\varphi_{F_1}(B_{\delta_{F_1}}(0)\cap B_{\delta_{F_2}}(0)),f\mapsto F_1-F_2 + f
	\een
	which is obviously $C^\infty$. Recall that $B_\delta(0)$ denotes the open ball of radius $\delta$ at the origin in $C^1(\overbar\VR^d,\VR^d)$.
	
	It remains to show that $U_F\subset \Cf$ is indeed open.  Let 
	$F_0=F+f_0\in U_F$, $f_0\in C^1(\overbar\VR^d,\VR^d)$ be given. Notice that by definition of the set $U_F$ we have $\|f_0\|_{C^1}<\delta_F$. Therefore $\hat \delta_{f_0} :=\delta_F - \|f_0\|_{C^1}$ is positive. Let $\eps>0$ be arbitrary. We need to show that there is $\eps>0$, such that $\|G-F\|_{C^1}<\delta_F$ for all $G\in \Cf$ with $d(F_0,G)<\eps$.  
	Let $G\in \Cf$ be any element satisfying $d(F_0,G)<\eps$. 	The fact that $F_0$ is a homeomorphism gives us
	$
	\|F_0-G\|_\infty + \|(\partial F_0-\partial G) (\partial F_0)^{-1}\|_\infty	= \|\Id-G\circ F_0^{-1}\|_{C^1}.
	$
	The definition of $d(\cdot,\cdot)$ and Lemma~\ref{lem:continuity_group} (as in the proof of Lemma~\ref{lem:convergence_F_n}) yield
	\ben\label{eq:eps_e_1}
 \|\Id-G\circ F_0^{-1}\|_{C^1} \le \eps e^\eps.
	\een
	Therefore we can choose $\eps>0$ so small that $\|G-F_0\|_{C^1} < \hat \delta_{f_0}$. Then
	\ben
	\|G-F\|_\infty \le  \underbrace{\|G-F_0\|_\infty}_{< \delta_F-\|f_0\|_\infty} 
	+ \|f_0\|_\infty < \delta_F
	\een
	and similarly by choosing  $\eps>0$ so small that $\|G-F_0\|_{C^1} < \tilde \delta_{f_0}/\| \partial F_0 \|_\infty$ we achieve the estimate, 
	\ben
	\begin{split}
		\|\partial G-\partial F\|_\infty  & \le \|\partial G - \partial F_0\|_\infty + 
		\|\partial f_0\|_\infty \\
		& \le \underbrace{\| \partial F_0 \|_\infty\|(\partial G - \partial 
			F_0)(\partial F_0)^{-1}\|_\infty}_{< \delta_F - \|\partial 
			f_0\|_\infty} + \|\partial f_0\|_\infty < \delta_F.
	\end{split}
	\een
	We conclude that if $\eps$ is so small that $\|G-F_0\|_{C^1} < \min\{\tilde \delta_{f_0}/\| \partial F_0 \|_\infty, \tilde \delta_{f_0}\}$, then the $\eps$-ball around $F_0$ in the $d$-topology is contained in $U_F$ and hence $U_F$ is open in $(\Cf,d)$.
\end{proof}

\section{Structure of first and second derivatives and shape Hessians}
This section is devoted to the structure of first and second order 
derivatives that were previously studied in \cite{zol79,MR2350816,MR1033071,MR1625559,MR1152456}.  
First we recall structure theorems giving the structure of the first and second derivative. Then we turn our attention to the structure of the  symmetric part of the second derivative as it is of great importance for our Newton method; \cite{MR1930612}.  We present a new proof of the structure theorem of the 
symmetric part by a successive application of the first and second structure theorem. The novelty of our approach is to connect all structure theorems with each other.

\subsection{Definition of first  and second derivatives}
The following definition recalls the standard notion of derivative of 
shape functions using the perturbation of identity. 
For given set $\Dsf\subset \VR^d$ we denote by $\wp(\Dsf)$ the powerset of $\Dsf$. We restrict ourselves to shape functions $J$ defined on $\mathcal A_\omega :=  \mathcal Z(\omega)\cap \wp(\Dsf ) $, where $\mathcal Z(\omega)$ was defined in \eqref{eq:image_set_X}. Notice that if $\omega$ is only of class $C^1$, then the elements in $\mathcal A_\omega$ are only of class $C^1$. However, we sometimes assume that a set in $\mathcal A_\omega$ is more regular for in which case we silently assume that $\omega$ is more regular.

In this section let $\omega\subset \Dsf$ be a bounded $C^1$ domain.

\begin{definition}\label{def1} 
	Let $J:\mathcal A_\omega\rightarrow \VR$ be a shape function and take $\Omega\in \mathcal A_\omega$. Let  $X,Y \in \ac C^1(\overbar\Dsf ,\VR^d)$ be two vector fields.
	\begin{itemize} 
		\item[(i)]  The directional derivative of $J$ at $\Omega$ in direction $X$ is defined by
		\ben
		DJ(\Omega)(X):= \lim_{t \to 0}\frac{J((\Id+tX)(\Omega))-J(\Omega)}{t}.
		\een
		\item[(ii)]   The second directional derivative of
		$J$ at $\Omega$ in direction $(X,Y)$ is defined by
		\ben
		\mathfrak D^2J(\Omega)(X)(Y) = \lim_{t \to 0}\frac{DJ((\Id+tY)(\Omega))(X\circ (\Id+tY)^{-1})-DJ(\Omega)(X)}{t},
		\een      
		($DJ((\Id+tY)(\Omega))(X\circ (\Id+tY)^{-1})$ exists for all small $t$).
		\item[(iii)]  If the directional derivative $DJ((\Id+tY)(\Omega))(X)$ exists for all small $t$, then  the second Euler derivative of $J$ at $\Omega$ in direction $(X,Y)$ is defined by
		\ben
		D^2J(\Omega)(X)(Y) = \lim_{t \to 0}\frac{DJ((\Id+tY)(\Omega))(X)-DJ(\Omega)(X)}{t}.
		\een
	\end{itemize}
\end{definition}

The following definition is concerned with the shape differentiability which we define as Hadamard semi-differentiability; see \cite[pp. 471]{DelZol11}.

\begin{definition}\label{def1_1_new} 
	Let $J:\mathcal A_\omega\rightarrow \VR$ be a shape function and let $\Omega\in \mathcal A_\omega$.
	\begin{itemize} 
		\item[(i)]  We say that $J$ is differentiable at $\Omega$ if 
			\ben\label{eq:hadamard_first_order}
			D_HJ(\Omega)(X)= \lim_{\substack{t \to 0\\ V\rightarrow X \textbf{ in  }C^1 }}\frac{J((\Id+tV)(\Omega))-J(\Omega)}{t}
			\een
			exists for all $X\in \ac C^1(\overbar\Dsf,\VR^d)$ and $X\mapsto 	D_HJ(\Omega)(X)$ is linear and continuous on $\ac C^1(\overbar\Dsf,\VR^d)$. 
		\item[(ii)] We say $J$ is twice differentiable at  $\Omega$ if it is differentiable in a neighborhood of $\Omega$, and if
		\begin{itemize}
			\item[$\bullet$] the mapping	$
			(X,Y)\mapsto D_HJ((\Id+Y)(\Omega))(X\circ (\Id+Y)^{-1})
			$
			is continuous at all $(X_0,0)\in (\ac C^1(\overbar \Dsf,\VR^d))^2$.
			\item[$\bullet$]  for all $X,Y\in \ac C^1(\overbar\Dsf,\VR^d)$ the limit
			\ben
			\mathfrak	D^2_HJ(\Omega)(X)(Y) = \lim_{\substack{t \to 0\\ W\rightarrow Y \textbf{ in  }C^1 }}\frac{D_HJ((\Id+tW)(\Omega))(X\circ (\Id+tW)^{-1})-D_HJ(\Omega)(X)}{t}
			\een
			exists, $(X,Y)\mapsto\mathfrak	D_H^2J(\Omega)(X)(Y)$ is bi-linear and continuous on  $(\ac C^1(\overbar \Dsf,\VR^d))^2$.
		\end{itemize} 
	\end{itemize}
\end{definition}

Recall that $\Phi_t^X$ denotes the flow of a vector field $X$.
\begin{lemma}\label{lem:hadamard_derivative_euler}
Assume that $J$ is differentiable at $\Omega\in \mathcal A_\omega$. Then we have
\ben
D_HJ(\Omega)(X) = \lim_{t\to 0} \frac{J(\Phi_t^X(\Omega))-J(\Omega)}{t} 
\een
for all $X\in \ac C^1(\overbar \Dsf,\VR^d)$.  
\end{lemma}
\begin{proof}
Setting $X_t := (\Phi_t^X -\Id)/t$ we can write $\Phi_t^X = \Id + t X_t$. Since $X_t \rightarrow X$ in $\ac C^1(\overbar \Dsf,\VR^d)$ as $t\to 0$, we obtain	
\ben
\begin{split}
D_HJ(\Omega)(X) & = \lim_{\substack{t\to 0\\ V\rightarrow X \text{ in  }C^1}} \frac{J((\Id+tV)(\Omega)) - J(\Omega)}{t} = \lim_{t\to 0} \frac{J((\Id+tX_t)(\Omega)) - J(\Omega)}{t}\\
&  = \lim_{t\to 0} \frac{J(\Phi_t(\Omega))-J(\Omega)}{t}. 
\end{split}
\een

\end{proof}

\begin{lemma}\label{lem:flow_vs_identity}
Let $J$ be differentiable at  $\Omega\in \mathcal A_\omega$ and assume that $\partial \Omega$ is of class $C^1$. Then
\ben
D_HJ(\Omega)(X) = 0\quad \text{ for all } X\in \ac C^1(\overbar\Dsf,\VR^d) \text{ satisfying } X\cdot \nu =0 \text{ on } \partial \Omega,
\een
where  $\nu$ denotes the outward pointing unit normal vector field along $\partial \Omega$.	
\end{lemma}	
\begin{proof}
Let $X\in \ac C^1(\overbar \Dsf,\VR^d)$ be such that $X\cdot \nu=0$ on $\partial \Omega$.  Then Nagumo's theorem \cite{MR0015180} shows
  $\Phi_t^X(\Omega)= \Omega$ for all $t$ and our claim follows from Lemma~\ref{lem:hadamard_derivative_euler}. 
 \end{proof}
%

\begin{example}\label{ex:J_f}
 As an illustration of the previous definition consider
	$
	J(\Omega) = \int_{\Omega} \fsf\; dx, 
	$
	where $\Omega\in\mathcal A_\omega$ is bounded and open.  This example can be found in \cite[pp. 28--29 ,Example 2.37]{sturm2015shape}.	 
	If $\fsf\in C^1(\overbar \Dsf)$, then $J$ is differentiable at $\Omega$ with derivative 
	in direction $X\in \ac C^1(\overbar \Dsf,\VR^d)$ given by
	\ben
	DJ(\Omega)(X) = \int_{\Omega} \VS_1:\partial X + \VS_0\cdot X \; dx, \qquad \VS_1(x) := \fsf(x)I,\; \VS_0(x) := \nabla \fsf(x).
	\een
	Here $:$ denotes the inner product on the space of matrices $\VR^{d,d}$ defined for 
	$A=(a_{ij}),B=(b_{ij})\in \VR^{d,d}$ by
	$
	A:B = \sum_{i,j=1}^d a_{ij}b_{ij}.
	$
	 Notice that for all small $t$ and $X,Y\in\ac C^1(\overbar \Dsf,\VR^d)$,
	\ben\label{eq:transform}
	DJ((\Id+tY)(\Omega))(X\circ(\Id+tY)^{-1} ) = \int_{\Omega}\det(I+t\partial Y)( \VS_1\circ (\Id+tY): \partial X(I+t\partial Y)^{-1} + \VS_0\circ (\Id+tY) \cdot X)\; dx.
	\een
	As a result if $\fsf$ belongs to $C^2(\overbar \Dsf)$, then $J(\cdot)$ is twice differentiable 
	at $\Omega$ with derivative
	$$   \mathfrak D^2J(\Omega)(X)(Y) =  \int_\Omega T_1(X): \partial Y + T_0(X)\cdot Y \; dx, $$
	where  $X,Y\in \ac C^1(\overbar \Dsf,\VR^d)$ and
	$  T_1(X) := (\fsf\Div(X) + \nabla \fsf \cdot X) I - \partial X^\top \fsf$, $T_0(X) := \nabla^2 \fsf X + \Div(X)\nabla \fsf.$
	Notice that  $  D^2J(\Omega)(X)(Y)
	$  exists for all $X,Y\in \ac C^2(\overbar \Dsf,\VR^d)$ and is given $
	D^2J(\Omega)(X)(Y) = \mathfrak D^2J(\Omega)(X)(Y) +  DJ(\Omega)(\partial XY).
$
	This decomposition of the Euler derivative is well-known (see \cite{MR1033071}) and holds for all twice differentiable shape functions $J$. We recall the precise statement in Lemma~\ref{thm:decomposition_second_order}. 
\end{example}

\subsection{Quotient space and restriction mapping}
Let $k\ge 0$ be an integer.
We introduce an equivalence relation on 
$ \ac C^k(\overbar\Dsf ,\VR^d)$ as follows: two vector fields
$X,Y\in  \ac C^k(\overbar\Dsf ,\VR^d) $ are equivalent, written $X\sim Y$, if and only if 
$ 
X=Y$ on $\partial \Omega 
$.  In other words two vector fields are equivalent if their
restriction ot $\partial \Omega$ coincides.
We denote the set of equivalence classes and its elements by 
$Q^k(\partial \Omega)$
and $\llb\VV\rrb$, respectively.
We denote by $\mathfrak J_{\partial \Omega}^k$ the restriction mapping of vector field 
belonging to $\ac C^k(\overbar\Dsf ,\VR^d)$ to mappings $\partial \Omega\rightarrow \VR^d$, that is,  
$
\mathfrak J_{\partial \Omega}^k: C^k(\overbar\Dsf ,\VR^d) \rightarrow {\partial \Omega}^{\VR^d},\quad  X\mapsto X_{|\partial \Omega},
$
where $\partial \Omega^{\VR^d}$ denotes the space of all mappings from $\partial \Omega$ into $\VR^d$.
The mapping $\mathfrak J_{\partial \Omega}$ induces the mapping 
$\tilde{\mathfrak{J}}_{\partial \Omega}^k: Q^k(\partial \Omega) \rightarrow {\partial \Omega}^{\VR^d}$ and 
by definition $\mathfrak J_{\partial \Omega}^k =  \tilde{\mathfrak{J}}_{\partial \Omega}^k\circ \pi  $,  where $\pi$ denotes the canonical surjection mapping  a vector field  $X\in \ac C^k(\overbar \Dsf ,\VR^d)$ to its equivalence class $\llb X \rrb$ in   $Q^k(\partial \Omega)$. We denote by  $\text{im}(\tilde{\mathfrak{I}}_{\partial \Omega}^k):= \{ \tilde{\mathfrak{I}}_{\partial \Omega}^k(X)|\; X\in Q^k(\partial \Omega)\}$ the image of $\tilde{\mathfrak{I}}_{\partial \Omega}^k$.

\subsection{First structure theorem}
The following theorem provides
the structure of the first (shape) derivative of a shape function $J$. 
\begin{theorem}\label{thm:structure_theorem-1}
	Let $\Omega\in \mathcal A_\omega$ be given and assume that $J$ is differentiable at $\Omega$.  Then:
	\begin{itemize}
		\item[(i)]  There is a linear  mapping 
		$\tilde  \Fg: \text{im}(\tilde{\mathfrak{J}}_{\partial \Omega}^1) \rightarrow \VR$ such that
		\ben\label{stru:very_general}
		DJ(\Omega)(X)=\tilde \Fg(X_{|\partial \Omega})
		\een
		for all $X\in \ac C^1( \overbar\Dsf ,\VR^d)$. 
		\item[(ii)] If $ \Omega\in C^1$,   then  
		$\text{im}(\tilde{\mathfrak{J}}_{\partial \Omega}^1)  = C^1(\partial \Omega,\VR^d)$ and 
		$\tilde \Fg:C^1(\partial \Omega,\VR^d) \rightarrow \VR$ is
		a continuous functional. 
		\item[(iii)]  If $\Omega\in C^2$, then $\Fg(v):=\tilde \Fg(v \nu)$ is continuous on $C^1(\partial \Omega)$ and satisfies
		\ben\label{volume}
		DJ(\Omega)(X)=\Fg(X_{|\partial \Omega}\cdot \nu) \quad \text{ for all } X\in \ac C^1( \overbar\Dsf ,\VR^d).
		\een 
	\end{itemize}
\end{theorem}
\begin{proof}
	This is a version of the structure theorem from \cite{sturm15_structure}.  Part (i) and (ii) follow the lines of the 
	proof of \cite{sturm15_structure}.
\end{proof}

\subsection{Second structure theorem}
In this section we recall the second structure theorem that provides 
a structure of $D^2J(\Omega)$. For more information we refer to \cite{MR1625559,MR1930612} and 
\cite[pp. 501]{DelZol11}.
\begin{lemma}\label{lem:symmetry}
	Let $X,Y\in \ac C^1(\overbar\Dsf ,\VR^d)$ and $\tau>0$ be given. Assume that 
	$f(s,h):=J((\Id + sX +  h Y)(\Omega))$ is twice continuously differentiable on $U:=(-\tau,\tau)\times 
	(-\tau,\tau)$. Then
	\ben\boxed{
		\mathfrak D^2J(\Omega)(X)(Y) = \mathfrak D^2J(\Omega)(Y)(X).}
	\een
\end{lemma}
\begin{proof}
	This is a consequence of Schwarz's theorem. Particularly $f$ is twice 
	continuously differentiable 
	on $U$. 
\end{proof}
\begin{remark}
	If the function $f$, defined in Lemma~\ref{lem:symmetry}, is not twice continuously differentiable, then  
	 $\mathfrak D^2J(\Omega)$ may be nonsymmetric. Consider for instance
	$J(\Omega) =\int_\Omega \fsf \; dx$
	with $\fsf$ only twice differentiable on $\VR^d$. Then 
	$\nabla^2 \fsf(x)$ is not necessarily symmetric which may destroys the 
	symmetry of $\mathfrak D^2J(\Omega)$. 
\end{remark}
The following theorem is called second structure theorem as it provides the structure of  $ D^2J(\Omega)$ which 
was first observed in \cite{MR1033071}.
\begin{theorem}\label{thm:decomposition_second_order}
	  Assume that $J$ is twice differentiable at the open set $\Omega\in \mathcal A_\omega$. Then we have for all $X\in \ac C^2( \overbar{\Dsf} ,\VR^d)$ and $Y\in \ac C^1( \overbar{\Dsf} ,\VR^d)$,
	\ben\label{eq:decom}\boxed{
		D^2J(\Omega)(X)(Y) = \mathfrak D^2J(\Omega)(X)(Y) + DJ(\Omega)(\partial XY)}
	\een
	and
	\ben\label{eq:covariant}
	\lim_{t\rightarrow 0} DJ((\Id+tY)(\Omega))\left(\frac{X\circ (\Id + 
		tY)^{-1} - X}{t} \right) = -DJ(\Omega)(\partial XY).
	\een	
\end{theorem}

\begin{proof}[Proof of Theorem~\ref{thm:decomposition_second_order}]
	Let $X \in \ac C^2(\overbar\Dsf ,\VR^d)$ and $Y \in \ac C^1(\overbar\Dsf ,\VR^d)$ be given and set $X_t = (X-X\circ (\Id+tY))/t$. Then $X_t\rightarrow - \partial XY$ in $\ac C^1(\overbar\Dsf,\VR^d)$ as $t\to 0$.
	Since $J$ is twice differentiable and the second derivative is continuous, we get 
	\[
	\lim_{t\rightarrow 0} DJ((\Id+tY)(\Omega))\left(\frac{X\circ (\Id + 
		tY)^{-1} - X}{t} \right) =  \lim_{t\rightarrow 0} DJ((\Id+tY)(\Omega))\left(X_t\circ (\Id + 
		tY)^{-1} \right) = -DJ(\Omega)(\partial XY)
	\]
	which is \eqref{eq:covariant}. This in turn yields
	\begin{align*} 
	DJ(\Omega)(\partial XY) + \mathfrak D^2J(\Omega)(X)(Y)  =& - \lim_{t\rightarrow 0} DJ((\Id+tY)(\Omega))\left(\frac{X\circ (\Id + 
		tY)^{-1} - X}{t} \right)\\
	& +\lim_{t\rightarrow 0}\frac{DJ((\Id+tY)(\Omega))\left(X\circ (\Id + 
		tY)^{-1}\right)-DJ(\Omega)(X)}{t}\\
	  =&\lim_{t\rightarrow 0} \frac{DJ((\Id+tY)(\Omega))(X) - DJ(\Omega)(X)}{t}= D^2J(\Omega)(X)(Y).
	\end{align*}
\end{proof}

\subsection{Third structure theorem and a new proof}
 The structure of the symmetric part of the second derivative was already analysed in \cite{MR1930612}. 
The novelty of our approach lies in the way how we derive it. We deduce the structure of the 
symmetric part by successively applying the first and second structure theorem.

\paragraph{Notation}
In the following we use the notation 
$ X_\tau := X_{|_{\partial \Omega}} - (X_{|_{\partial \Omega}}\cdot \nu) \nu$ and 
$A_\tau := A_{|_{\partial \Omega}} - (A_{|_{\partial \Omega}} \nu)\otimes\nu$ to 
indicate the tangential part of the vector fields $X\in \ac C^1(\overbar\Dsf ,\VR^d)$ and $A\in C^1(\overbar\Dsf,\VR^{d,d})$ restricted to $\partial \Omega$.
Here $\nu$ is the outward pointing unit normal field along $\partial \Omega$ and  $\otimes$ denotes the 
tensor product defined by $(a\otimes b)c:= (c\cdot b)a$ for all $a,b,c\in \VR^d$. The tangential gradient of $f\in C^1(\partial \Omega)$ and Jacobian and divergence of  $g\in C^1(\partial \Omega,\VR^d)$ 
can then be defined by $\nabla^\tau f:= (\nabla \tilde f)_\tau$, $\partial^\tau g := (\partial \tilde g)_\tau$ and $\Div_\tau(g) := \partial^\tau \tilde g:I$, where $\tilde g,\tilde f$ are $C^1$ extensions of $g,f$ to a neighborhood of $\partial \Omega.$

\paragraph{Third structure theorem}
The following theorem will be referred to as third structure theorem.
\begin{theorem}\label{thm:structure_second}
	Let  $\Omega\in \mathcal A_\omega$  be an open set and let $J$  be twice differentiable at $\Omega$. 
	\begin{itemize}
		\item[(i)] There are mappings 	
		$\tilde \Fg:\text{im}(\mathfrak J_{\partial \Omega}^1) \rightarrow \VR $ and  $\tilde \Fl:\text{im}(\mathfrak J_{\partial \Omega}^1)\times \text{im}(\mathfrak 
		J_{\partial \Omega}^1)  \rightarrow \VR $, 
		such that
		\ben\label{eq:struc_second_zero}
		\begin{split}
			\mathfrak D^2J(\Omega)(X)(Y)  = \tilde \Fl(X_{|_{\partial \Omega}}, Y_{|_{\partial 
					\Omega}})\quad \text{ and } \quad DJ(\Omega)(\partial XY) =  \tilde \Fg((\partial XY)_{|_{\partial \Omega}})
		\end{split}
		\een
		and hence
		\ben\label{eq:struc_second_one}
		D^2J(\Omega)(X)(Y) = \tilde \Fl(X_{|_{\partial \Omega}}, Y_{|_{\partial 
				\Omega}}) + \tilde \Fg((\partial XY)_{|_{\partial \Omega}})
		\een
		for all $X,Y\in \ac C^2(\overbar\Dsf,\VR^d)$.
		\item[(ii)] If   $\partial \Omega \in C^2$, then $\text{im}(\mathfrak 
		J_{\partial \Omega}^1) = C^1(\partial \Omega,\VR^d)$ and $\tilde 
		\Fg $ and $\tilde 
		\Fl$ are continuous on $C^1(\partial \Omega,\VR^d)$. 
		\item[(iii)] If  $\partial \Omega\in C^3$,
		then $\Fg(v):=\tilde \Fg(v \nu)$ and $\Fl(v,w):=\tilde \Fl(v \nu,w\nu)$ are continuous on $C^1(\partial \Omega)$ and $(C^1(\partial \Omega))^2$, respectively and satisfy
		\ben\label{eq:struc_second_two}
		\begin{split}
		\mathfrak D^2J(\Omega)(X)(Y) =&  \Fl(X_{|_{\partial \Omega}}\cdot 
		\nu, Y_{|_{\partial \Omega}}\cdot \nu) - \Fg(\partial^\tau X_\tau Y_\tau\cdot \nu ) \\
		&-	\Fg(\nabla^\tau(Y\cdot \nu)\cdot X_\tau) - \Fg(\nabla^\tau(X\cdot \nu)\cdot Y_\tau)
		\end{split}
		\een
		 for all $X,Y\in \ac C^2(\overbar\Dsf,\VR^d)$. 
	\end{itemize}
\end{theorem}
\begin{proof}
	\emph{(i):}	
	Firstly on account of the differentiability assumption on $J$ and of Theorem~\ref{thm:decomposition_second_order} we have
$
	D^2J(\Omega)(X)(Y) = \mathfrak D^2J(\Omega)(X)(Y) + DJ(\Omega)(\partial XY)
$
	for all $X\in \ac C^2( \overbar{\Dsf} ,\VR^d)$ and  $Y\in \ac C^1( \overbar{\Dsf} ,\VR^d)$. Let $X\in \ac C^2( \overbar\Dsf ,\VR^d)$  and $Y\in\ac C^1( \overbar\Dsf ,\VR^d)$.  
	The Banach fixed point theorem 
	shows that $T_{t,s}:=\Id+sX+tY$ is bijective on $\VR^d$ for all $s,t$ small enough.
	Moreover if $X=Y=0$ on $\partial \Omega$, then
	$T_{t,s}(\Omega) = \Omega$ for all small $t,s$. Thus we have
	$\mathfrak D^2J(\Omega)(X)(Y) = 0$ for all $X\in \ac C^2( \overbar\Dsf ,\VR^d)$ and $Y\in \ac C^1( \overbar\Dsf ,\VR^d)$ with  $X=Y=0$ on $\partial \Omega$ and by density this yields $\mathfrak D^2J(\Omega)(X)(Y) = 0$ for all $X,Y\in \ac C^1( \overbar\Dsf ,\VR^d)$ with  $X=Y=0$ on $\partial \Omega$.	Hence the mapping
	$
	\mathfrak h(\llb X\rrb,\llb Y\rrb) := \mathfrak D^2J(\Omega)(X)(Y)
	$
	is well-defined for all $\llb X\rrb,\llb Y\rrb\in Q^1(\partial \Omega)$. Since $\tilde{\mathfrak{J}}_{\partial \Omega}^1$ is a bijection onto  $\text{im}(\tilde{\mathfrak{J}}_{\partial \Omega}^1)$, we can define 
	$\tilde \Fl(X,Y) := \mathfrak h((\tilde{\mathfrak{J}}_{\partial \Omega}^1)^{-1}(X),(\tilde{\mathfrak{J}}_{\partial \Omega}^1)^{-1}(Y) )$ which satisfies by definition
	\ben\label{eq:tilde_l}
	\tilde \Fl(X_{|_{\partial \Omega}}, Y_{|_{\partial \Omega}}) := \tilde \Fl(\tilde{\mathfrak{J}}_{\partial \Omega}(\llb X \rrb), \tilde{\mathfrak{J}}_{\partial \Omega}(\llb Y \rrb))  = \mathfrak h(\llb X \rrb,\llb Y \rrb) = \mathfrak D^2J(\Omega)(X)(Y)
	\een
	for all $X,Y\in \ac C^1( \overbar\Dsf ,\VR^d)$.  Finally by 
	the first structure theorem (Theorem~\ref{thm:structure_theorem-1}), we have
	$
	DJ(\Omega)(X) = \tilde \Fg(X_{|_{\partial \Omega}})
	$ for all  $X\in \ac C^1( \overbar\Dsf ,\VR^d)$ 
	and plugging this together with \eqref{eq:tilde_l} into \eqref{eq:decom} we recover \eqref{eq:struc_second_one} and also \eqref{eq:struc_second_zero}.
	\newline
	\emph{(ii)} This follows from the continuity of the extension operator $E:C^1(\partial \Omega,\VR^d)\rightarrow \ac C^1(\overbar\Dsf,\VR^d)$.
	\newline 
	\emph{(iii)}
	Note that since $\Omega$ is $C^2$, Theorem~\ref{thm:structure_theorem-1} item (iii) yields that 
	$\Fg(v) := \tilde \Fg(v\nu)$  is continuous on $C^1(\partial \Omega)$ and
	satisfies 
	$
	DJ(\Omega)(X) = \Fg(X_{|_{\partial \Omega}}\cdot \nu) 
	$ for all $X\in \ac C^1( \overbar\Dsf ,\VR^d)$.
	It follows from Lemma~\ref{lem:flow_vs_identity} that $D^2J(\Omega)(X)(Y)=0$ for all $X,Y\in \ac C^2( \overbar\Dsf,\VR^d)$ with $Y\cdot \nu =0$ on 
	$\partial \Omega$.  In view of \eqref{eq:struc_second_one} this yields $\tilde \Fl(X_{|_{\partial \Omega}}, Y_{|_{\partial 
			\Omega}})  = - \Fg((\partial X)_{|_{\partial \Omega}}Y_\tau\cdot \nu )$ for all $X,Y\in \ac C^2( \overbar\Dsf,\VR^d)$ with $Y\cdot \nu =0$ on 
	$\partial \Omega$. Since $(\partial X)_{|_{\partial \Omega}}Y_\tau = \partial^\tau X Y_\tau$ this is equivalent to the important equation
	\ben\label{eq:normal}
	\tilde \Fl(X_{|_{\partial \Omega}}, Y_{|_{\partial 
			\Omega}})  = - \Fg(\partial^\tau X Y_\tau\cdot \nu ) \quad \text{ for all } X,Y\in \ac C^2( \overbar\Dsf,\VR^d) \text{ with } Y\cdot \nu =0 \text{ on } \partial \Omega.
	\een
	Now let $X,Y\in \ac C^2( \overbar\Dsf  ,\VR^d)$ be arbitrary. Splitting the restrictions of $X,Y$ to $\partial \Omega$ into normal and tangential parts and inserting the results into \eqref{eq:normal} gives
	\ben\label{eq:splitting}
	\begin{split}
		\tilde \Fl(X_{|_{\partial \Omega}}, Y_{|_{\partial \Omega}}) &  = \tilde 
		\Fl((X\cdot \nu)\nu, (Y\cdot \nu)\nu) +  \tilde \Fl(X_\tau, Y_\tau) 
		+ \tilde \Fl((X\cdot \nu)\nu,Y_\tau) + \tilde \Fl(X_\tau, (Y\cdot \nu)\nu) 
		\\
		&= \tilde \Fl((X\cdot \nu)\nu, (Y\cdot \nu)\nu) - \Fg((\partial^\tau (X\cdot \nu \nu))Y_\tau\cdot \nu ) - \Fg((\partial^\tau (Y\cdot \nu \nu))X_\tau\cdot \nu ) - \Fg(\partial^\tau X_\tau Y_\tau\cdot \nu )
	\end{split}
	\een
	valid for all $X,Y\in \ac C^2(\overbar\Dsf,\VR^d)$.  Now notice that 
	$
	\partial^\tau ((X\cdot\nu)\nu) = \nu\otimes \nabla^\tau (X\cdot \nu) + (X\cdot \nu) \partial^\tau \nu, 
	$
	and hence
	\ben\label{eq:partial_X_1}
	\begin{split}
		\Fg(\partial^\tau ((X\cdot \nu) \nu)Y_\tau\cdot \nu )
		& = \Fg(\underbrace{ 
			(\nu\otimes \nabla^\tau (X\cdot \nu))Y_\tau \cdot \nu}_{= \nabla^\tau(X\cdot \nu)\cdot Y_\tau}) +  \Fg((X\cdot \nu)  \partial^\tau \nu Y_\tau\cdot 
			\nu )
	\end{split}
	\een
	and by interchanging the roles of $X$ and $Y$ also
	\ben\label{eq:partial_X_2}
	\Fg((\partial^\tau (Y\cdot \nu \nu))X_\tau\cdot \nu ) =  \Fg(\nabla^\tau(Y\cdot \nu)\cdot X_\tau)+\Fg((Y\cdot \nu)  \partial^\tau \nu X_\tau\cdot 
	\nu).
	\een
	Since $|\nu|=1$ on $\partial \Omega$ we get $\partial \nu^\top \nu=0$ on $\partial \Omega$. Multiplying with any tangent vector $\gamma_x\in T_x(\partial \Omega)$ yields 
	$0=\partial \nu^\top(x) \nu(x)\cdot \gamma_x = \nu(x)\cdot \partial \nu(x) \gamma_x = \nu(x)\cdot \partial^\tau \nu(x) \gamma_x$. This means 
	$\partial^\tau \nu(x)(T_x(\partial \Omega))\subset T_x(\partial \Omega)$ and hence $\partial^\tau \nu X_\tau\cdot 
	\nu=0$.
	Therefore inserting \eqref{eq:partial_X_1},\eqref{eq:partial_X_2} into \eqref{eq:splitting} gives us
	\ben
	\begin{split}
	\tilde \Fl(X_{|_{\partial \Omega}}, Y_{|_{\partial \Omega}})  = & \tilde \Fl((X\cdot 
	\nu)\nu, (Y\cdot \nu)\nu) - 
	\Fg(\nabla^\tau(Y\cdot \nu)\cdot X_\tau) - \Fg(\nabla^\tau(X\cdot \nu)\cdot Y_\tau) - \Fg(\partial^\tau X_\tau Y_\tau\cdot \nu )
	\end{split}
	\een
	Finally setting $\Fl(v,w):= \tilde \Fl(v\nu, w\nu) 
	$ we recover formula \eqref{eq:struc_second_two}. 
\end{proof}

\begin{remark}
	Notice that the second part of formula \eqref{eq:struc_second_two} can be rewritten by noting that
	\ben
	\begin{split}
	\partial^\tau X_\tau Y_\tau\cdot \nu & =  X_\tau \cdot \partial^\tau\nu Y_\tau\\
	\nabla^\tau(X\cdot \nu)\cdot Y_\tau & =  \nu \cdot \partial^\tau X Y_\tau +  X  \cdot \partial^\tau \nu Y_\tau = \partial^\tau X \nu \cdot Y_\tau.
	\end{split}
	\een
	Substituting this into \eqref{eq:struc_second_two} we obtain
	\ben
	\mathfrak D^2J(\Omega)(X)(Y) =  \Fl(X_{|_{\partial \Omega}}, Y_{|_{\partial 
			\Omega}}) - \Fg( \nu \cdot \partial^\tau Y X_\tau) - \Fg(\nu\cdot \partial^\tau X Y_\tau ) -  \Fg( Y_\tau \cdot \partial^\tau \nu X_\tau) 
	\een
	which is precisely equation (2.7) in \cite{MR1930612}. 	The function $\partial^\tau \nu : T(\partial \Omega) \to T(\partial \Omega)$ is sometimes called shape operator or Weingarden map. 
\end{remark}

\subsection{Boundary and domain Hessian}\label{sec:first_second}
\paragraph{Definition the shape Hessians}
 We now define a two shape Hessians.
\begin{definition}
	Let $\Omega\in \mathcal A_\omega$ be given and assume that $J$ is twice differentiable at $\Omega$.  The domain shape Hessian $H_{\Omega,J}^{\text{vol}}:\ac C^1(\overbar\Dsf,\VR^d)\times \ac C^1(\overbar\Dsf,\VR^d)\rightarrow \VR$  of $J$ at $\Omega$ is defined by
	\ben
	H_{\Omega,J}^{\text{vol}}(X)(Y) := \mathfrak D^2J(\Omega)(X)(Y) \quad (= D^2J(\Omega)(X)(Y)- DJ(\Omega)(\partial XY)).
	\een
Let $\tilde \nu\in \ac C^1(\overbar \Dsf,\VR^d)$ be a $C^1$-extension of the outward pointing unit normal field $\nu$ along $\partial \Omega$ and set $X_\nu := (X\cdot \tilde \nu)\tilde \nu$. 	The \textit{boundary shape Hessian} $H_{\Omega,J}^{\text{bry}}:\ac C^1(\overbar\Dsf,\VR^d)\times \ac C^1(\overbar\Dsf,\VR^d)\rightarrow \VR$  at $\Omega$ is defined by 
		\ben\label{def:hess_quotient}
		H_{\Omega,J}^{\text{bry}}(X)(Y) :=	H_{\Omega,J}^{\text{vol}}(X_\nu)(Y_\nu),
		\een
		for all $X,Y\in \ac C^2(\overbar\Dsf,\VR^d)$.
\end{definition}

\begin{remark}
	\begin{itemize}
		\item[(i)]  Within our framework both Hessians are \textit{symmetric}.  For functions $J$ defined on $\Cf$ the domain Hessian corresponds to the Hessian on 
		the manifold $\Cf$. It only depends on the Euclidean connection $(X,Y)\mapsto \partial XY$. 
		The canonical Hessian on the quotient $\Cf/\Cg_\omega$ is the boundary Hessian, which only depends on the Euclidean connection. 
		\item[(ii)] Our approach is based on the metric group $\Cf$ associated with the vector space $C^1(\overbar \VR^d,\VR^d)$. However,  other vector spaces to construct a metric group, e.g. the space of bounded and Lipschitz continuous function $C^{0,1}(\overbar \VR^d,\VR^d)$, are possible. 
		Since the metric group $\Cf$ is contained in an affine space $\Id+\Theta$, where $\Theta$ equals e.g. $C^{0,1}(\overbar \VR^d,\VR^d)$ or $C^1_b(\overbar \VR^d,\VR^d)$, the tangent space of the corresponding metric group $\Cf(\Theta)$ is always $\Theta$; see \cite[Theorem 2.17, p.151]{MR2731611}. 
	\end{itemize}
\end{remark}

The boundary Hessian is defined as the restriction of the domain Hessian to normal perturbations.  Hence we have
$
H_{\Omega,J}^{\text{vol}}(X)(Y)  = H_{\Omega,J}^{\text{bry}}(X)(Y)$ for all $ X,Y $ with $ X_\tau = Y_\tau=0$ on $\partial \Omega$.
Moreover if $J$ satisfies the assumptions of Theorem~\ref{thm:structure_second}, then 
\[
H_{\Omega,J}^{\text{bry}}(X)(Y) = \Fl(X_{|_{\partial \Omega}}\cdot \nu, Y_{|_{\partial \Omega}}\cdot \nu).
\]

\paragraph{Example of boundary and domain shape Hessians}
Let us briefly revisit the shape function $J(\Omega) := \int_{\Omega}\fsf \; dx$, where $\fsf\in C^2(\overbar \Dsf)$ and $\Omega\in \mathcal A_\omega$ is open and bounded. 
In Example~\ref{ex:J_f} we computed the domain shape Hessian of $J$, namely  
\ben\label{eq:hess_volume}
H_{\Omega,J}^{\text{vol}}(X)(Y) =  \int_\Omega T_1(X): \partial Y + T_0(X)\cdot Y \; dx, 
\een
where  $X,Y\in \ac C^1(\overbar\Dsf,\VR^d)$ and
$  T_1(X) := (\fsf\Div(X) + \nabla \fsf \cdot X) I - \partial X^\top\fsf$ and  $T_0(X) = \nabla^2 \fsf X + \Div(X)\nabla \fsf.$ Following the steps of the proof of \cite[Lemma 3.11]{lauraindistributed} we can readily bring \eqref{eq:hess_volume} into the boundary form \eqref{eq:struc_second_two}. 
Since $\mathfrak D^2J(\Omega)(X)(Y) = 0$ for all $X\in \ac C^2(\overbar\Dsf,\VR^d)$ with $\supp(X)\subset \Omega$, we conclude by partial integration $-\Div(T_1(X))+ T_0(X)=0$ 
everywhere in $\Omega$. This in turn shows by partial integration  
$\mathfrak D^2J(\Omega)(X)(Y)   = \int_{\partial \Omega}  T_1(X)\nu \cdot Y\; ds$ for all 
$X\in \ac C^2(\overbar\Dsf,\VR^d)$.	Recall that $\Div_\tau(X)=\Div_\tau(X_\tau) + \kappa X\cdot \nu$, where $\kappa := \Div_\tau(\nu)$ is the mean curvature of $\partial \Omega$. Then by splitting the restrictions of $X,Y$ to $\partial \Omega$ into normal and tangential part and assuming $\partial \Omega$ is of class $C^2$ we check, 
\ben\label{eq:second_example_boundary}
\begin{split}
	\mathfrak D^2J(\Omega)(X)(Y)  = &\int_{\partial \Omega}  T_1(X)\nu \cdot Y\; ds =  \int_{\partial \Omega} (\fsf \Div(X)+\nabla \fsf\cdot X)(Y\cdot \nu)  -\fsf \nu \cdot \partial XY\; ds\\
	 = & \int_{\partial \Omega} (\fsf \kappa +\nabla f\cdot \nu ) (X\cdot \nu)(Y\cdot \nu)\; ds \\
	&+ \int_{\partial\Omega} \fsf \Div_\tau(X_\tau) Y\cdot \nu + \nabla^\tau\fsf \cdot X (Y\cdot \nu) + \fsf \partial X\nu \cdot \nu (Y\cdot \nu)- \fsf \nu \cdot \partial XY\; ds.
\end{split}
\een
 Using the tangential Stokes formula \cite[p.498]{DelZol11} we obtain
\ben\label{eq:stokes}
\int_{\partial \Omega} \fsf  (Y\cdot \nu)\Div_\tau(X_\tau) \; ds = - \int_{\partial \Omega} X_\tau \cdot \nabla^\tau\fsf (Y\cdot \nu) + \fsf X_\tau \cdot\nabla^\tau(Y\cdot \nu)\;ds.
\een
Further by splitting $X,Y$ into normal and tangential part,
\ben\label{eq:stokes2}
 \nu \cdot \partial XY = \nu\cdot \partial^\tau XY_\tau +  \partial X\nu\cdot \nu (Y\cdot \nu) = \nu\cdot \partial^\tau X_\tau Y_\tau + \nu\cdot \partial^\tau \nu Y_\tau (X\cdot \nu) +  \nabla^\tau (X\cdot \nu)\cdot Y_\tau +  \partial X\nu\cdot \nu (Y\cdot \nu).
\een
 Plugging \eqref{eq:stokes} and \eqref{eq:stokes2} into \eqref{eq:second_example_boundary} and using $\nu\cdot \partial^\tau \nu Y_\tau=0$  we obtain
\ben\label{eq:second_example_boundary2}
\begin{split}
	\mathfrak D^2J(\Omega)(X)(Y)  = &\int_{\partial \Omega} (\fsf \kappa +\nabla \fsf \cdot \nu ) (X\cdot \nu)(Y\cdot \nu)\; ds - \int_{\partial \Omega} \fsf\nu\cdot \partial^\tau X_\tau Y_\tau\;ds\\
	& -\int_{\partial \Omega}\fsf (X_\tau \cdot\nabla^\tau(Y\cdot \nu)+ Y_\tau \cdot\nabla^\tau(X\cdot \nu))\;ds.
\end{split}
\een
Notice that  \eqref{eq:second_example_boundary} has the predicted form \eqref{eq:struc_second_two}. From \eqref{eq:second_example_boundary} we also see that the boundary shape Hessian is given by
\ben\label{eq:hess_bnd}
H_{\Omega,J}^{\text{bry}}(X)(Y)= \int_{\partial \Omega}  (\nabla \fsf\cdot \nu + \kappa\fsf) (X\cdot \nu)(Y\cdot \nu)\; ds, \quad \text{ for } X,Y\in \ac C^2(\overbar\Dsf,\VR^d).
\een

\begin{remark}[Positive definiteness]\label{rem:pos_def}
	As a conclusion of the previous example we see that the boundary Hessian of $J$ will be positive 
	definite if $\nabla \fsf \cdot \nu +\fsf\kappa>\eps$ on $\partial \Omega$ for some 
	constant $\eps >0$. Then
	$
	H_{\Omega,J}^{bry}(X)(X)\ge \eps \|X\cdot \nu\|^2_{L_2(\partial \Omega)}
	$ for all $X\in \ac C^1(\overbar \Dsf,\VR^d)$. However, the boundary Hessian $H^{\text{bry}}_{\Omega,J}$ does not need to be positive definite in a stationary 
	point $\Omega^*$. 	Indeed consider $\fsf\in C^3(\VR)$ given by
	\ben
	\fsf(x) := \left\{\begin{array}{cc}
		(x-1)^4    & \text{ if } x >1, \\
		0 & \text{ if } x\in [-1,1],\\
		-(x+1)^4       & \text{ if } x<1
	\end{array} \right..
	\een
	Then $\Omega^* =(-1,1)$  is a stationary point of $J(\Omega) = \int_{\Omega} \fsf \; dx$, but also the second derivative vanishes at $\Omega^*$. 
\end{remark}

\subsection{Newton's equation on $\Cf$ and $\Cf/\Cg_\omega$}
Let $J$ be a twice differentiable shape function on $\mathcal A_\omega$ and take any $\Omega\in \mathcal A_\omega$. 
\begin{definition}
We call $g^{vol} \in C^1(\partial \Omega,\VR^d)$ a domain Newton direction at $\Omega$ if 
$g^{vol} = g|_{\partial \Omega}$ and   $g\in \ac C^1(\overbar \Dsf,\VR^d)$, solves
\ben\label{eq:domain_newton}
H_{\Omega,J}^{\text{vol}}(g)(Y) = - DJ(\Omega)(Y) \quad \text{ for all } Y\in \ac C^1(\overbar \Dsf,\VR^d).
\een	
We call $g^{bry}\in \ac C^1(\partial \Omega,\VR^d)$ a boundary Newton direction at $\Omega$ if
\ben\label{eq:boundary_newton}
H_{\Omega,J}^{\text{bry}}(g^{bry})(Y) = - DJ(\Omega)(Y) \quad \text{ for all } Y\in \ac C^1(\overbar \Dsf,\VR^d).
\een
\end{definition}

The task of the next section is to construct a finite dimensional subspace of $\ac C^1(\overbar \Dsf,\VR^d)$ on which \eqref{eq:domain_newton} can be solved. 

\begin{example}\label{ex:boundary_domain_hessian}
 We already computed  the boundary Hessian \eqref{eq:hess_bnd} and the domain Hessian \eqref{eq:hess_volume} for the shape functional $J(\Omega) = \int_{\Omega} \fsf \; dx$ at $\Omega\in \mathcal A_\omega$. It is readily seen that the restriction of a volume Newton direction for this example is of the form $g^{vol} = f/(\nabla f\cdot\nu + f\kappa)\nu +  Z_\tau$, $Z\in C^1(\partial \Omega,\VR^d)$. The boundary Newton direction is given by $g^{bry} = (f/(\nabla f\cdot\nu + f\kappa)){\tiny }\nu$. 
\end{example}

\section{Approximate normal basis functions}
This section is devoted to the construction of basis functions, called approximate normal functions,  with which we aim to 
discretise the Newton equation \eqref{eq:domain_newton}. The idea is to construct vector fields  that are linearly independent and additionally "normal enough" to domain of interest such that the discrete Hessians can be inverted. 

The main ingredient for our 
construction are symmetric positive definite kernels and more specifically positive definite radial kernels. Positive definite and symmetric kernels generate reproducing kernel Hilbert spaces (RKHS)
which are characterised by the property that the point evaluation is a continuous functional. They allow to work with 
the reproducing kernel instead of the RKHS itself. For instance shape gradients may be computed explicitly as shown in  \cite{eigelsturm16} without solving a boundary value problem

Throughout this section $M\subset \VR^d$ is a $C^1$-submanifold of codimension one and we denote by $\nu_M$ a normal field along $M$. 

\subsection{Reproducing kernel Hilbert spaces}

We begin with the definition of matrix-valued  reproducing  kernels. 
\begin{definition}
	Let $\mathcal X\subset \VR^d$ be an arbitrary set. A function 
	$\Ksf:\mathcal X\times \mathcal X\rightarrow \VR^{d,d}$ is called 
	\emph{matrix-valued reproducing kernel} for the Hilbert space 
	$\mathcal H(\mathcal X,\VR^d)$ of functions $f:\mathcal X\rightarrow \VR^d$, if for all $x\in \mathcal X, a\in \VR^d$ and $f\in \Hrep$,
	\begin{itemize}
		\item[$(a)$] 
		$
		\Ksf(x,\cdot)a\in \Hrep
		$
		\item[$(b)$] 
		$		
		(\Ksf(x,\cdot)a, f)_{\Hrep} = a\cdot f(x).	
		$
	\end{itemize}
	In case $d=1$ we call $\Ksf$ scalar reproducing kernel and in order to 
	 distinguish  the matrix and scalar case we set 
	$\ksf(x,y):=\Ksf(x,y)$ and $\mathcal H(\mathcal X):=\mathcal H(\mathcal X,\VR^1)$. 	
\end{definition}
\begin{remark}
	\begin{itemize}
		\item 
		Notice that in case $d=1$ the items (a) and (b) of the previous definition  read:  for all $x\in \mathcal X$ and $f\in \mathcal H(\mathcal X)$  we have 
			$\ksf(x,\cdot)\in \mathcal H(\mathcal X)$, and 
			$(\ksf(x,\cdot), f(\cdot))_{\mathcal H(\mathcal X)} = f(x).	
			$
		\item Notice that items (a) and (b) together imply that the point evaluation $\delta_x(f):=f(x)$ is a continuous functional on a reproducing kernel Hilbert space.
	\end{itemize}
\end{remark}
The following remark collects a few interesting properties of reproducing 
Hilbert spaces; cf.\cite{MR2131724}.
\begin{remark}\label{rem:kernel}
	\begin{itemize}
		\item It is readily checked that a (scalar) reproducing kernel is symmetric, 
		$\ksf(x,y)=\ksf(y,x)$ for all $x,y\in \mathcal X$.  It is also  positive semi-definite, that is, 
		for all mutually distinct $\{x_1,\ldots, x_N\}$ the matrix $(\ksf(x_i,x_j))$ is 
		positive semidefinite. When this latter matrix is positive definite 
		for all mutually distinct $x_i$ we call $\ksf$ positive definite reproducing 
		kernel. If a kernel $\ksf$ is positive definite then for all mutually distinct points
		$\{x_1,\ldots, x_M\}\subset \mathcal X$,  $M\ge 1$,   the functions $\{\ksf(x_1,\cdot), \ldots \ksf(x_M,\cdot)\}$ are linearly independent.  
		\item Let $\mathcal X=\Omega$, $\Omega\subset \VR^d$ open, and $\ksf(x,\cdot)\in C(\Omega)$ for all $x\in \Omega$. Then we have the inclusion $\mathcal H(\Omega)\subset C(\Omega)$; cf. \cite[pp.133]{MR2131724}. 
		\item When we start with a scalar reproducing kernel $\ksf$ on $\mathcal X\subset \VR^d$ with 
		RKHS $\mathcal H(\mathcal X)$, then $\Ksf(x,y):=\ksf(x,y)I$ is a matrix-valued reproducing kernel with RKHS $[\mathcal H(\mathcal X)]^d$. Moreover, the inner product is given by
		$
		(f,g)_{\mathcal H(\mathcal X,\VR^d)} := (f_1,g_1)_{\mathcal H(\mathcal X)} + \cdots + (f_d,g_d)_{\mathcal H(\mathcal X)}	
		$
		for all $f=(f_1,\ldots, f_d)$ and $g=(g_1,\ldots, g_d)$ with $f_1,\ldots, f_d, g_1,\ldots, g_d\in \mathcal H(\mathcal X)$.  A proof can be found in \cite{eigelsturm16}.
	\end{itemize}
\end{remark}

\begin{example}
	An example of positive definite kernel is the Gaussian 
	kernel $\ksf^\sigma(x,y) := e^{-\frac{|x-y|^2}{\sigma}}$, $\sigma>0$; cf. \cite{MR3044643}. Another important compactly supported radial kernel that is positive definite is $\ksf^\sigma(x,y):= (1-\frac{|x-y|}{\sigma})_+^4(4\frac{|x-y|}{\sigma} +1)$, $\sigma>0$; \cite[pp.119]{MR2131724}.	 
\end{example}

\subsection{Approximate normal basis functions}
We now define special basis function on  $M\subset\VR^d$. These new 
basis functions are vector fields $\VR^d\rightarrow \VR^d$ with the 
pleasing property that their restriction to $M$ is approximately normal in a certain sense (cf. Lemma~\ref{lem:kernel_function_small}). 

\begin{definition}[ Normal and approximate basis functions]\label{def:basis_functions}
	 Let  $\ksf:\VR^d\times \VR^d \rightarrow \VR$ be a positive definite reproducing kernel.
	\begin{itemize}
		\item[(a)]
		  We define the \emph{approximate 
		normal basis function} $\Fv^x=\Fv^x_M:\VR^d\rightarrow \VR^d$ associated with the point $x\in M$ by 
	\ben\label{eq:basis_vi}
	\Fv^x(y) := \nu_M(x) \ksf(x,y).
	\een
	For an arbitrary set $\mathcal X\subset M$ we define	
	the approximate normal space
	$
	\VXM := \overline{\text{span}\{\Fv^x(\cdot):\; x\in \mathcal X \} },
	$
	where the closure is taken in $[\mathcal H(\VR^d)]^d$, the vvRKHS associated with the scalar kernel $\ksf$. In case
	$\mathcal X=M$ we set $\VMM:=\VXM$. 
	
	\item[(b)]	We define the normal function $\Fw^x=\Fw^x_M:\VR^d\rightarrow \VR^d$ associated with the point $x\in M$ by 
		\ben\label{eq:normal_basis}
		\Fw^x(y) := \nu_M(y) \ksf(x,y)
		\een
	and the normal space,
	 $
	 \NXM := \overline{\text{span}\{\Fw^x(\cdot):\; x\in \mathcal X \} }.
	 $
	 The closure is taken in the vvRKHS $[\mathcal H(M)]^d$ associated with the restriction of $\ksf$ to $M$.
	\end{itemize}
\end{definition}
Whenever no confusion is possible we simply write $\Fv^x$ (resp. $\Fw^x$) instead of $\Fv^x_{M}$ (resp. $\Fw^x_M$). 
Notice that we have the inclusion $\VMM \subset C(\VR^d,\VR^d)$ as 
$[\mathcal H(\VR^d)]^d \subset C(\VR^d,\VR^d)$; cf. Remark~\ref{rem:kernel}.

\begin{figure}
	\centering
	
	\begin{tikzpicture}[allow upside down, scale=1.5]
	
	\draw [black,line width=1pt,smooth, tension=1.9,samples=1000,rounded 
	corners=15pt] (2.86,6.4) -- node[sloped,inner sep=0cm,above,pos=.5, 
	anchor=south west, minimum height=1cm,minimum width=0.5cm](N){} 
	(3.62,7.5) -- (6.42,8.4) -- (8,7.5)-- (8.4,6.4); 
	\node (solution) at ($(5,7)$) {$\Omega$};
	\node (solution) at ($(2.8,8)$) {$\Dsf\setminus \Omega$};
	
	[]
	\centerarc[pattern=north west lines, pattern color = red!50!black,thick](5.85,8.20)(0:360:0.2);

	\path[every node/.style={font=\sffamily\small}]
	(5.9,8.25) edge[ <-] node [left] {} (6.6,9.25);
	\node (solution) at ($(6.5,9.45)$) {{\small point $x$ }};
	
	\fill[red!50!black] (5.85,8.20) circle (1pt);
	
	\end{tikzpicture}
	\caption{Sketch of a basis function $\Fv^x(y) = \phi_\sigma(|x-y|)\nu_M(x)$ that has 
		support (dashed red) around the point $x
		$.  From the picture it can be observed that $\Fv^x(y)\approx 
		\phi_\sigma(|x-y|)\nu_M(y)$ for all $y$ near $x$ and 
		$\Fv^x(y)=0$ for all $y$ far away from $x$. Here the submanifold is $M=\partial \Omega$. }
	\label{fig:sketch_basis}
\end{figure}
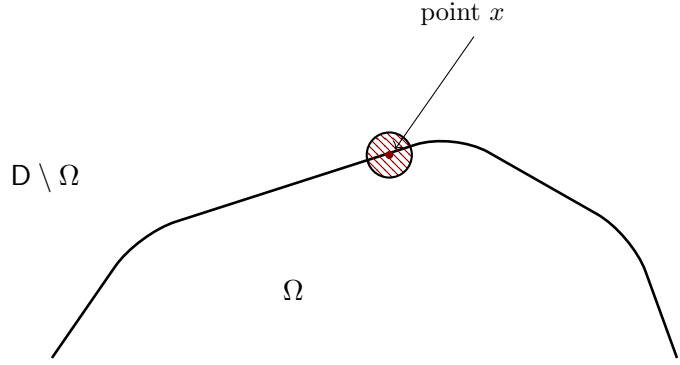

\subsection{Inner products on approximate normal spaces}
In this subsection let $M$, $\ksf$ and $\VMM$ be defined as 
in Definition~\ref{def:basis_functions}.

\begin{lemma}\label{lem:linearly_independent}
	The vector fields   
	$\{\Fv^{x_1},\ldots, \Fv^{x_N}\} $ defined in \eqref{eq:basis_vi} are linearly 
	independent if and only if  $\{x_1,\ldots, x_N\}\subset \mathcal X$ are pairwise 
	distinct. 
\end{lemma}
\begin{proof}
	Let $\alpha_1,\ldots, \alpha_N\in \VR$ be such that $
	\sum_{i=1}^N \alpha_i \Fv^i(x) = 0  
	$ for all $x\in \VR^d$.
	Since $\{\ksf(x_1,x), \ldots, \ksf(x_N,x)) \}$ are 
	linearly independent on $\VR^d$, we obtain
	$\alpha_1 \nu_M(x_1) = \cdots =\alpha_N \nu_M(x_N)=0$. But at each point $x_i$ one component of $\nu_M(x_i)$ must be non-zero since 
	 $|\nu_M(x_i)|=1$  and hence we conclude $\alpha_i=0$ for $i=1,\ldots, N$.  
\end{proof}

Next we compute the orthogonal complement of  
$\VXM$ in $[\mathcal H(\VR^d)]^d$. 

\begin{lemma}\label{lem:tangential_Vn}
	We have for arbitrary subset $\mathcal X\subset M$,
	\ben\label{eq:orth_comp}
	\VXM^\bot = \{f\in [\mathcal H(\VR^d)]^d :\; f(x)\cdot \nu_M(x)=0 
	\quad \text{ for all } x\in \mathcal X\}.
	\een
\end{lemma}
\begin{proof}
Let use denote by $\nu_M^\ell$ the components of the vector field $\nu_M$. 	We have for every $f=(f_1,\ldots, f_d)\in [\mathcal H(\VR^d)]^d$ and $x\in \mathcal X$,
	\ben\label{eq:orth_v}
	(f, \Fv^x)_{[\mathcal H]^d} = \sum_{\ell=1}^d (f_\ell, 
	\nu_M^\ell(x)\ksf(x,\cdot ))_{\mathcal H} = \sum_{\ell=1}^d \nu_M^\ell(x) f_\ell(x) = \nu_M(x)\cdot f(x),
	\een
	where in the penultimate step we used the reproducing property of 
	$\ksf$. 
	
	Let us now show the inclusion $\subset$ in \eqref{eq:orth_comp}. Let $f\in \VXM^\bot$ be arbitrary. Then in view of \eqref{eq:orth_v}  we get
	$0=(f, \Fv^x)_{\mathcal H} = \nu_M(x)\cdot f(x)$ for all $x\in \mathcal X$,  so that $f\in \{f\in [\mathcal H(\VR^d)]^d :\; f(x)\cdot \nu_M(x)=0 
	\quad \text{ for all } x\in \mathcal X\}$. 
	It remains to prove $\supset$. Let $f\in[\mathcal H(\VR^d)]^d$ be such that  $f(x)\cdot \nu_M(x)=0
	$ for all $x\in \mathcal X$. Then again in view of \eqref{eq:orth_v} for all $x\in \mathcal X$,
	$
	(f, \Fv^x)_{[\mathcal H]^d} = f(x)\cdot \nu_M(x)=0.
	$
	By linearity and density we conclude
	$(f, \Fv)_{[\mathcal H]^d} =0$ for all $\Fv\in\VXM$ which shows $f\in \VXM^\bot$ and 
	finishes the proof. 
\end{proof}

The previous lemma tells us that $f(x)\cdot \nu_M(x)=0$ for 
all $f\in \VXM^\bot$ and all $x\in \mathcal X$.  
However it is not true that $f(x)\cdot \nu_M(x)=0$ 
for all $x\in M$.
But the (possibly uncountable) number of tangential points of $f\in \VXM^\bot$ at 
$M$ increases with the dimension of 
$\VXM$.

\begin{lemma}\label{lem:tangential_Vn2}
	Let $X\in \mathcal V^M(\VR^d,\VR^d)$ be such that $X\cdot \nu_M =0$ on $M$. Then $X=0$ on $\VR^d$.	
\end{lemma}
\begin{proof}
	If $X\cdot \nu_M=0$ on $M$, then Lemma~\ref{lem:tangential_Vn} shows  $X\in \VMM^\bot$. 
	Since  $[\mathcal H(\VR^d)]^d = \VMM^\bot \oplus \VMM$ we must have $X=0$. 	
\end{proof}
\begin{remark}
	Assume that $M$ is compact. 	
	Then the functions
	\begin{align}
		(X,Y)_{L_2,M} & := \int_{M} (X\cdot \nu_M) (Y\cdot \nu_M)\; ds,	\label{eq:inner} \\
		(X,Y)_{H^1,M} &:= \int_{M} \nabla^\tau(X\cdot \nu_M)\cdot  \nabla^\tau(Y\cdot \nu_M) + (X\cdot \nu_M) (Y\cdot \nu_M)\; ds \label{eq:inner3}
		\end{align}
	define inner products on $\VMM$. Hence $\VMM$ equipped with \eqref{eq:inner} or \eqref{eq:inner3} is a pre-Hilbert space. 
	\begin{proof}
	It is also clear that the functions defined in \eqref{eq:inner} and \eqref{eq:inner3} are bilinear and non-negative. 
	It remains to check that $(X,X)=0$ if and only if $X=0$. 
	If $X=0$, then it is obvious that $(X,X)=0$. Since $\|X\|_{L_2,M}\le \|X\|_{H^1,M}$ we only need to show the converge statement for $(\cdot,\cdot)_{L_2,M}$. For $X\in \VMM$ the equality  $(X,X)_{L_2,M}=0$ is equivalent to  $X\cdot \nu_M=0$ on $M$.  Hence  Lemma~\ref{lem:tangential_Vn2} implies $X=0$ on $\VR^d$. 
\end{proof}
\end{remark}
The previous remark shows that $(\VMM, (\cdot, \cdot))$ is a pre-Hilbert space with $(\cdot, \cdot)$ given by \eqref{eq:inner} or \eqref{eq:inner3}, which is not necessarily complete.  Finally let us mention \cite{MR1998617} and also \cite{doi:10.1137/15M1029369} for the discussion of other interesting  metrics including $H^{-1/2}$ and $H^{1/2}$ Sobolev-type metrics.

\subsection{Basis function of radial kernels}\label{subsec:basis_radial}
Let us now examine how "normal" the fields in $\VMM$ actually are. Recall that for a function $f:M\rightarrow \VR^d$ the tangential part is defined by $f_\tau := f- (f\cdot \nu_M)\nu_M$.  Throughout the rest of the paper we assume:
\begin{assumption}\label{ass:phi}
	Let  $\phi \in C^2([0,\infty])$ and $\supp\; \phi \subset [0,1]$.
\end{assumption}  

For $\sigma>0$ we associated with $\phi$ the radial kernel $\ksf^\sigma(x,y) = \phi(|x-y|/\sigma)$, $x,y\in \VR^d$ and the function $\hat \phi(x):=\phi(|x|)$, $x\in \VR^d$. We readily check that there are constants $c_1,c_2>0$, so that for all $\sigma>0$, $
|\ksf^\sigma(x,y)|\le c_1$ and also $|\nabla_y \ksf^\sigma(x,y)| \le \frac{c_2}{\sigma}$ for all $x,y\in \VR^d.$

An example of a positive definite function $\phi$ in $\VR^2$ satisfying Assumption~\ref{ass:phi} is given by $\phi(r) = c(1+r)^4_+(4r+1)$ for some positive constant $c$; see \cite[pp.129]{MR2131724} and also \cite{MR1664547,MR1357720,MR1366510}. For these radial
kernels 
it is possible to explicitly determine their native space, i.e., the
Hilbert space they generate.

\begin{lemma}\label{lem:kernel_function_small}
	Assume that	$M$ is compact and let $\phi$ satisfy Assumption~\ref{ass:phi}. 
	Set $\ksf^\sigma(x,y) = \phi(|x-y|/\sigma)$ and $\Fv^x_\sigma(y):=\ksf^\sigma(x,y)\nu_M(x)$.
	For  every $x\in M$, we have
	\ben\label{eq:est_one}
	\lim_{\sigma \searrow 0} \|(\Fv^x_\sigma)_\tau\|_{C(M,\VR^d)} = 0.
	\een
	If $M $ is of class $C^2$, then there are constants $c_1,c_2>0$, so that for $x\in M$,
	$
	\|\nabla^\tau(\Fv^x_\sigma\cdot \nu_M)\|_{C(M,\VR^d)} \le c_1 + c_2/\sigma$ for all $\sigma >0$.
\end{lemma}
\begin{proof}
	Since $\nu_M$ is continuous on $M$ and $|\nu_M|=1$ on $M$, we find for every $x\in M$ and every 
	$\eps >0$ a number 
	$\delta >0$ so that $|\nu_M(x)-\nu_M(y)|<\eps$ and $|1-\nu(x)\cdot \nu_M(y)|<\eps$ for all $y\in  M$ 
	with $|x-y|<\delta.$  Define $L:=\max_{r\in \VR}|\phi(r)|$, then 
	$|\ksf^\sigma(x,y)|\le L$ for all $x,y\in \VR^d$ and all $\sigma >0$.  
	Now for all $y\in M$ with $|x-y|<\delta$ we get the estimate
	\ben\label{eq:estimate_v_x}
	\begin{split}
		|(\Fv^x_\sigma)_\tau(y)|& = 	|\Fv^x_\sigma(y)-(\Fv^x_\sigma(y)\cdot \nu_M(y))\cdot \nu_M(y)|  =
		|\ksf^\sigma(x,y)\nu_M(x) - \nu_M (x)\cdot (\ksf^\sigma(x,y) \nu_M(y))\nu_M(y)|\\
		& \le \underbrace{|\ksf^\sigma(x,y)}_{\le L}| \underbrace{|\nu_M(x)-\nu_M(y)|}_{\le 
			\eps} + 
		\underbrace{|\ksf^\sigma(x,y)|}_{\le 
			 L }\underbrace{|\nu_M(y)|}_{=1}\underbrace{|1-\nu_M(x)\cdot 
			\nu_M(y)|}_{\le \eps} \le  2L  \eps.  
	\end{split}
	\een
	In view of $\supp(\phi)\subset [0,1]$ we have 
	$\Fv^x_\sigma(y)=0$ for all $x,y\in M$ with 
	$|x-y|> \sigma$. As a consequence \eqref{eq:estimate_v_x} is valid for all 
	$y\in M$ when $ \sigma <\delta$ and thus for all  
	$ \sigma <\delta$ we have
	$
	\|(\Fv^x_\sigma)^\tau\|_{C( M ,\VR^d)} \le  2L \eps . 
	$
	This shows that for arbitrary $\eps >0$ we find $\delta >0$ so that 
	$\|(\Fv^x_\sigma)^\tau\|_{C( M ,\VR^d)} \le  2L \eps$   for all 
	$ \sigma <\delta$ which shows \eqref{eq:est_one}.
	
	Let $M$ be of class $C^2$.  By assumption  $|\nabla_y \ksf^\sigma(x,y)| \le \frac{c}{\sigma}$
	for all $x, y$ and $\sigma>0$. Therefore 
	 $
	\nabla^\tau(\Fv^x_\sigma\cdot \nu_M) = (\partial^\tau \Fv^x_\sigma)^\top\nu_M + (\partial^\tau\nu_M)^\top \Fv^x_\sigma =  (\nu_M(x)\otimes \nabla_y^\tau\ksf^\sigma(x,y)))^\top \nu_M + (\partial^\tau\nu_M)^\top \Fv^x_\sigma
	$   
	and this shows $\nabla^\tau(\Fv^x_\sigma\cdot \nu_M)$ is bounded by $c_1 + c_2/\sigma$ and finishes the prove.
\end{proof}

\subsection{Transport of approximate normal basis functions}

In the following we set $\hat \phi (x) := \phi(|x|/\sigma)$ for a fixed positive number $\sigma$, where $\phi$ satisfies Assumption~\ref{ass:phi}. Recall that $B_\delta(0)$ denotes the open ball in $C^1(\overbar\VR^d,\VR^d)$ centered at the origin of radius $\delta>0$. Let $q\in (0,1)$. 
Set $T_s^g := \Id+sg$ for $s\in [0,1]$ and $g\in B_q(0)$.  
Given  $\hat \phi\in C(\VR^d)$, we define for all $x\in M$,  $s\in(0,1)$,
\ben\label{eq:transport_2} 
\mathfrak T^{s,g}_M:\hat \phi(x-\cdot)\nu_M(x)\mapsto \hat \phi(x^s-\cdot) \nu_{T^{g,s}(M)}(x^s),
\een
where $x^s := T^{g,s}(x)$. In other words the transport $\mathfrak T^{g,s}$ maps the approximate normal function associated with the point $x$ at $M$ to the approximate normal function associated with the point $x^s$ at $T^{s,g}(M)$. In view of $\nu_{T^{g,s}(M)} = \nu^{g,s}_M \circ (T^{g,s})^{-1}$ with $\nu^{g,s}_M(x) := \frac{(\partial T^{g,s}(x))^{-\top}\nu_M(x)}{|(\partial T^{g,s}(x))^{-\top}\nu_M(x)|}$ the transport reads
\ben
 \mathfrak T^{s,g}_M(\hat \phi(x-\cdot)\nu_M(x)) = \hat \phi(T^{g,s}(x)-\cdot) \nu^{g,s}_M(x).
\een

\begin{lemma}\label{lem:transport}
 Let $q\in (0,1)$ and $x\in M$ be given.
Set $\Fv_{M,s}^{g,x} := \mathfrak T^{g,s}(\hat\phi(x-\cdot)\nu_M(x))$, $g\in B_q(0)$. Then $s\mapsto \Fv_{M,s}^{g,x}:[0,1]\rightarrow C^1(\overbar \VR^d,\VR^d)$ is differentiable and its derivative $\dot \Fv^{g,x}_{M,s}(y):= \frac{d}{ds} \Fv^{g,x}_{M,s}(y)$ is given by 
\ben\label{eq:formula_dot_v_s}
 	\begin{split}
 		\dot \Fv^{g,x}_{M,s}(y)  = \big(\nu^{g,s}_M(x)\otimes \nabla \hat \phi\left(T^{g,s}(x)-y\right)\big) g(x)   + \hat \phi\left(T^{g,s}(x)-y\right) \dot \nu^s_M(x),
 	\end{split}
 	\een
 	where $\dot \nu^{g,s}_M = - (\partial T^{g,s})^{-\top}\partial g^{\top} \nu^{g,s}_M + \nu^{g,s}_M (\nu^{g,s}_M\cdot (\partial T^{g,s})^{-\top}\partial g^{\top }\nu^{g,s}_M)$.
 	There is a constant $c>0$, independent of $x$, such that for all $g\in B_q(0)$ and all $s\in [0,1]$
 	\ben\label{eq:estimate_dot_v}
 	\|\dot \Fv_{M,s}^{g,x}\|_{C^1} \le c\|g\|_{C^1}.
 	\een
 	Moreover, we have
 	\ben
 	\partial \Fv_{M,s}^{g,x}(y) g(y) = - \bigg[\nu^{g,s}_M(x)\otimes \nabla \hat \phi\left(T^{g,s}(x)-y\right)\bigg]g(y).
 	\een
\end{lemma}
\begin{proof}
	
The function $f(s,x) := \Fv_{M,s}^{g,x}$ is of class $C^2$ since $\nu^s_M$ and $\hat \phi$ are of class $C^2$. It follows that $s\mapsto \Fv_{M,s}^{g,x}, [0,1] \rightarrow C^1(\overbar \VR^d,\VR^d)$ is differentiable.  Formula \eqref{eq:formula_dot_v_s} follows by direct computation and this shows \eqref{eq:estimate_dot_v}.
\end{proof}

\begin{corollary}\label{corr:esitmate_kernel_parallel_transport}
Let the  hypotheses of the previous lemma be satisfied. Let  $q\in (0,1)$ be given. Then there is a constant $c>0$ so that for all  $x\in M$ and  $g\in B_q(0)$, 
\ben\label{eq:lip_v_x}
\|\Fv^x_M-\Fv^{x+g(x)}_{(\Id+g)(M)}\|_{C^1}\le c\|g\|_{C^1}.
\een	
\end{corollary}
\begin{proof}
	Estimate \eqref{eq:lip_v_x} follows directly from the fundamental theorem of calculus applied to $s\mapsto \Fv_{M,s}^{x,g}$ and \eqref{eq:estimate_dot_v}.
\end{proof}

\section{Newton's method for shape functions $J(\Omega)$}
This section is devoted to the convergence analysis of  a Newton algorithm 
in the spirit of \cite{MR2893875}. The Newton equation 
will be solved in the approximate normal space using the basis functions introduced in the previous section. 
We  prove  the convergence of Newton's method in the discrete setting, however, an analog in the finite dimensional setting should also hold under suitable conditions. We work with the domain shape Hessian $H^{\text{vol}}_{\Omega,J}=\mathfrak D^2J(\Omega)$ restricted
to a finite dimensional subspace of $\VOmega$ which is an approximation of the boundary shape Hessian $H^{\text{bry}}_{\Omega,J}$.

\subsection{Setting and algorithm}
Let a bounded $C^1$ domain $\omega_0\subset \VR^d$, a finite 
number of points $\mathcal X=\{x_1,\ldots,x_n \}$ a finite number
of points contained in $\partial \omega_0$, and a twice
differentiable shape function $J$ on $\mathcal A_{\omega_0}$ be given. 

Our Newton method reads: find $g_k \in \text{span}\{\Fv^1_k,\ldots, \Fv^n_k \}$ such that
\ben\label{eq:newton_equation}
H_{F_k(\omega_0),J}^{vol}(g_k)(\varphi) = - DJ(F_k(\omega_0))(\varphi)\quad \text{ 
	for all } \varphi\in \text{span}\{\Fv^1_k,\ldots, \Fv^n_k \}. 
\een
We set $F_0:=\Id$ and update $F_k$ by setting $F_{k+1} := T_k^1 \circ F_k$, where $T_k^s :=\Id + sg_k$, $s\in[0,1]$. The basis functions $\Fv_k^i, i=1,\ldots, n$ are given by $\Fv_k^i(y):=\Fv_{k,0}^i(y)$, where
\[
\Fv^i_{k,s} := \hat \phi(T_k^s(x_k^i)-\cdot)\nu_k^s(x_k^i), \quad \Omega_k^s := T^s_k(\Omega_k), \quad \nu_k^s := \frac{(\partial T_k^s)^{-\top}\nu_k}{|(\partial T_k^s)^{-\top}\nu_k|}, \quad x_k^{i,s} := T_k^s(x_k^i).
\]
 We also set $\Fv^{i,s}_k := \Fv^i_{k,s}\circ T_k^s$ and $\dot\Fv^{i,s}_k = \frac{d}{ds}\Fv^{i,s}_k$. By the chain rule,
\ben\label{eq:material_derivative}
(\Fv^i_{k,s})':= \frac{d}{ds} (\Fv^{i,s}_k\circ (T_k^s)^{-1}) = ( \dot\Fv^{i,s}_k)\circ (T_k^s)^{-1} - (\partial \Fv^{i,s}_k(\partial T_k^s)^{-1}g_k)\circ (T_k^s)^{-1}
\een
since $ \frac{d}{ds} (T_k^s)^{-1} = -((\partial T_k^s)^{-1}g_k)\circ (T_k^s)^{-1}$.

It is convenient to write \eqref{eq:newton_equation} in matrix form. 
For this purpose set $\Omega_k :=F_k(\omega_0)$ and introduce the following notation for the discrete domain Hessian and first derivative,
\ben\label{eq:newton_discrete}
 H_k   := (H_{\Omega_k,J}^{vol}(\Fv_k^i)(\Fv_k^j))_{i,j=1,\ldots, n} , \quad  l_k   := (DJ(F_k(\Omega))(\Fv_k^i))_{i=1,\ldots, n}.
\een
Further we set  
$\mathcal X_k = \{F_k(x_1),\ldots, F_k(x_n)\}$.  At the $k$th iteration we  identify the Euclidean space $\VR^n$ with  
$\text{span}\{\Fv_k^1,\ldots, \Fv_k^n\}$ via
$
 P_k: \; (y_1,\ldots, y_n)\mapsto \sum_{\ell=1}^n y_\ell \Fv^\ell_k.
$
It satisfies  $\|g\|_{C^1} = \|P_k(X)\|_{C^1}\le \|P_k\| |X|$ for all $X\in \VR^n$.
Now we can write \eqref{eq:newton_equation} in matrix notation as follows method
\ben\label{eq:newton_discrete2}
H_k X_k = -l_k, \quad F_{k+1} = (\Id+g_k) \circ F_k, \quad g_k= P_k(X_k).
\een

We consider the following algorithm.

\begin{algorithm}[H]
	
	
	
	\KwData{Let  $\gamma>0$ and $n, N\in \VN$ be given. Choose  
		$\Omega\subset \VR^d$ and 
		$\mathcal X_0 := \{x_1,\ldots, x_n\}\subset \Omega$. Let 
		$F_0:=\Id$. }

	initialization\;
	
	\While{ $k  \le  N $}{
		
		1.) Compute $X_k\in \VR^n$ as solution of
		$
		H_{\mathcal X_k}(F_k)X_k = -L_{\mathcal X_k}(F_k).
		$
		Set $g_k := P_{\mathcal X_k}(X_k)$.
		
		2.) Update $F_{k+1} \gets (\Id+g_k)\circ F_k$.

		3.)     Update $\mathcal X_{k+1} \gets \{F_{k+1}(x_1),\ldots, F_{k+1}(x_n)\}$.
		
		4.)         Update $\Omega_{k+1} \gets F_{k+1}(\Omega_k)$.
		
		5.)         Update $v^{F_{k+1}(x_i)}_{\partial\Omega_{k+1}} \gets  v^{F_k(x_i)}_{\partial\Omega_k}  $.	
		
		\eIf{ $J(F_k(\Omega)) - J(F_{k+1}(\Omega)) \ge \gamma (J(\Omega) - J(F_1(\Omega)))$ }{            
			step accepted: continue program\; 
		}{
		no sufficient decrease: quit\;
	}            
	increase $k \gets k+1$\;
}
\caption{Newton algorithm} \label{eq:algo_newton}
\end{algorithm}

\subsection{Convergence analysis of Newton's method for shape functions $J(\Omega)$}

Subsequently we need the following auxiliary result. 
\begin{lemma}\label{lem:contraction}
Let $(a_k)$ be a sequence of nonnegative numbers. Let $c>0$ be a constant and let $q_2,q_1\in (0,1)$ be two 
numbers satisfying $\bar q:=q_1+q_2 <1$. Assume
\ben\label{eq:ak}
a_{k+1} \le c a_k^2 + q_2a_k \quad \text{ for all }  k\ge 0.
\een	
If the initial number $a_0$ is such that $ca_0<q_1$, then
\ben
a_{k+1} \le  \bar q^{k+1} a_0 \quad \text{ for all } k\ge 0
\een
and consequently $(a_k)$ goes to zero.
\end{lemma}
\begin{proof}
The proof follows easily  by induction over $k$.
\end{proof}

Now we are in a position to show show that $g_k$ converges to 
zero in $C^1(\overbar\VR^d,\VR^d)$ and $(F_k)$ converges to some element $F_*$ in $\mathcal \Cf$. With the setting and notation from the previous paragraph we now prove the following theorem. 

\begin{theorem}\label{thm:main_Newton}
Let $J$, $\omega_0$, $F_k$ and $g_k$ as before. 	Assume there is $q\in (0,1)$ such that $g_k\in B_q(0)$ for all $k\ge 0$. Moreover, let the following hypothesis be satisfied for all $i=1,\ldots, n$. 
	\begin{itemize}
		\item[(A1)] The matrix $H_k$ is invertible and there is $c>0$, such that
				$
				\|H_k^{-1}\|\le c 
				$
				for $ k\ge 0 $.	
		\item[(A2)] There is a constant $c>0$, such that for all $k\ge 0$  and  $s\in [0,1]$,
		\ben\label{eq:A1}
		|D^2J(T_k^s(\Omega_k))(\Fv_k^{i,s})(g_k\circ (\Id+sg_k)^{-1}) -  DJ^2(T_k^0(\Omega_k))(\Fv_k^{i,0})(g_k)| \le c\|g_k\|_{C^1}^2.
		\een
		\item[(A3)] There is a constant $c>0$, such that for all $k\ge 0$  and  $s\in [0,1]$,
		\ben\label{eq:A2}
		|DJ(T_k^s(\Omega_k))((\partial \Fv_k^{i,s} (\partial T_k^s)^{-1}g_k)\circ (T_k^s)^{-1})-  DJ(\Omega_k)(\partial \Fv_k^{i,0}g_k)| \le c\|g_k\|_{C^1}^2.
		\een

		\item[(A4)] There is a sequence $(p_k)$, $p_k\in [0,\tilde q]$, $\tilde q\in (0,1)$, such that for all $k\ge 0$ and  $s\in [0,1]$,
		\ben\label{eq:A3}
				\sum_{i=1}^n\|H_k^{-1}\|\|P_k\| |DJ(T_k^s(\Omega_k))(\dot \Fv_k^{i,s}\circ (T_k^s)^{-1}) | \le p_k\|g_k\|_{C^1}.
				\een

	\end{itemize}
\text{ }\\[0.1cm]	
 \noindent
	Then there holds:
	\begin{itemize}
		\item[(i)] There is a constant $c>0$, such that the series $\kappa_k :=p_k + c|X_k|$ satisfies
		\ben\label{eq:convergence_rate}
		|X_{k+1}|  \le \kappa_k|X_k| \quad\text{ for all } k\ge 0.
		\een
         If $|X_0|c+\tilde q<1$, then $X_k \to 0$ as $k\to 0$. 
		\item[(ii)] 
		Under the conditions of (i) there is an element $F_*\in \Cf$, such that
		$
		d(F_*,F_k) \rightarrow 0
		$
		as $k\rightarrow \infty$ and we have an estimate
		\ben\label{eq:apriori}
		d(F_*,F_k) \le  5\|g_0\|_{C^1}\frac{\alpha^k}{1-\alpha } \quad\text{ for all } k\ge 0. 
		\een 
		Moreover if $
		F\mapsto \mathfrak D^2J(F(\Omega)): \Cf\rightarrow \mathcal L(C^1(\overbar\VR^d,\VR^d),\mathcal 
		L(C^1(\overbar\VR^d,\VR^d),\VR)) 
		$
		and $
		F\mapsto DJ(F(\Omega)): \Cf \rightarrow \mathcal L(C^1(\overbar\VR^d,\VR^d),\VR)
		$
		are continuous at $F_*\in \Cf$, then
		$DJ(F_*(\Omega))(\Fv_*^i)=0$ for $i=1,\ldots, n$, where $\Fv_*^i$ denotes the approximate normal function associated with $F_*(\Omega)$  and the point $F_*(x_i)$. If $DJ(F_*(\Omega))(\varphi)=0$ for all $\varphi \in C^1(\overbar \VR^d,\VR^d)$, then $p_k\to 0$ as $k\to \infty$ and hence $\kappa_k\to 0$ as $k\to \infty$. In this case the sequence $(X_k)$ converges superlinearly to zero. 
		
	\end{itemize}
\end{theorem}
\begin{proof}
(i) For $k\ge 0$, let $g_k\in \text{span}\{\Fv^1_k,\ldots, \Fv^n_k \}\cap  B_q(0)$  be the solution of the Newton equation
\ben\label{eq:newton_proof} 
\mathfrak D^2J(\Omega_k)(g_k)(\varphi) = - DJ(\Omega_k)(\varphi) \quad \text{ for all } \varphi\in \text{span}\{\Fv^1_k,\ldots, \Fv^n_k \}.
\een
  Since $J$ is twice differentiable,  Theorem~\ref{thm:decomposition_second_order} yields,
\ben\label{eq:decom_secon_proof}
D^2J(\Omega_k)(X)(Y) = \mathfrak D^2J(\Omega_k)(X)(Y) +  DJ(\Omega_k)(\partial XY) 
\een
for all $X,Y\in C^2(\overbar \VR^d,\VR^d)$.
Hence inserting $\Fv^i_{k,0}$ as test function into \eqref{eq:newton_proof} and using \eqref{eq:decom_secon_proof} yield,
\ben\label{eq:first_derivative}
DJ(\Omega_k)(\Fv^i_{k,0}) \stackrel{\eqref{eq:newton_proof} }{=} - \mathfrak D^2J(\Omega_k)(g_k)(\Fv^i_{k,0}) \stackrel{\eqref{eq:decom_secon_proof} }{=} - D^2J(\Omega_k)(\Fv^i_{k,0})(g_k) + DJ(\Omega_k)(\partial \Fv^i_{k,0} g_k). 
\een
   According to Lemma~\ref{lem:transport} the function $s\mapsto \Fv^i_{k,s},[0,1]\to C^1(\overbar \VR^d,\VR^d)$ is differentiable for $i=1,\ldots, n$. 
 Hence an application of the fundamental theorem of calculus to $s\mapsto DJ(T_k^s(\Omega))(\Fv^i_{k,s})$ on $[0,1]$ yields
\ben\label{eq:proof_newton_1}
DJ(T_k^1(\Omega_k))(\Fv^i_{k,1}) = DJ(T_k^0(\Omega_k))(\Fv^i_{k,0}) + \int_0^1 D^2J(T_k^s(\Omega_k))(\Fv^i_{k,s})(g_k\circ (\Id+sg_k)^{-1})\;ds + \int_0^1 DJ(T_k^s(\Omega_k))((\Fv^i_{k,s})') \;ds.
\een
 It is readily checked that $\Fv^i_{k,1} = \Fv^i_{k+1}$ and $\Fv^i_{k,0} = \Fv^i_k$.
Therefore using \eqref{eq:first_derivative} we can rewrite \eqref{eq:proof_newton_1} in the equivalent form
\ben\label{eq:proof_newton_12}
\begin{split}
DJ(T_k^1(\Omega_k))(\Fv^i_{k+1})  = & \int_0^1 D^2J(T_k^s(\Omega_k))(\Fv^i_{k,s})(g_k\circ (\Id+sg_k)^{-1}) -  DJ^2(T_k^0(\Omega_k))(\Fv^i_{k,0})(g_k)\;ds \\
&+ \int_0^1 DJ(T_k^s(\Omega_k))((\Fv^i_{k,s})') \;ds + DJ(\Omega_k)(\partial \Fv^{i,0}_kg_k).
\end{split}
\een
Using \eqref{eq:material_derivative} the previous equation reads
		\ben\label{eq:proof_newton_123}
		\begin{split}
			DJ(T_k^1(\Omega_k))(\Fv_{k+1}^i)  = & \int_0^1 D^2J(T_k^s(\Omega_k))(\Fv_k^{i,s})(g_k\circ (\Id+sg_k)^{-1}) -  DJ^2(T_k^0(\Omega_k))(\Fv_k^{i,0})(g_k)\;ds \\
			&+ \int_0^1 DJ(T_k^s(\Omega_k))(\dot \Fv_k^{i,s}\circ (T_k^s)^{-1}) \;ds \\
			& -\int_0^1DJ(T_k^s(\Omega_k))((\partial \Fv_k^{i,s} (\partial T_k^s)^{-1}g_k)\circ (T_k^s)^{-1})-  DJ(\Omega_k)(\partial \Fv_k^{i,0}g_k) \; ds.
		\end{split}
		\een
This shows, using (A2)-(A4), that there is $c\ge 0$, such that
$
|DJ(T_k^1(\Omega_k))(\Fv^i_{k+1})| \le p_k\|P_k\|^{-1}\|g_k\|_{C^1} + c \|g_k\|_{C^1}^2
$
and hence
\ben\label{eq:estimate_l}
|(l_{k+1})_i| = |DJ(T_k^1(\Omega_k))(\Fv^i_{k+1})|  \le p_k|X_k| + c|X_k|^2.
\een 
Since  $X_{k+1}$ solves the Newton equation $H_{k+1} X_{k+1} = -l_{k+1}$ we get using the boundedness of $H_{k+1}^{-1}$ and  \eqref{eq:estimate_l} that
	$
	|X_{k+1}| = |H_{k+1}^{-1}l_{k+1}| \le c|l_{k+1}| \le  p_k|X_k| + c|X_k|^2 
	$
	for all $k\ge 0$. Hence we may apply Lemma~\ref{lem:contraction} with $a_k :=|X_k|$, $q_2:=\tilde q$ and  $\bar q := c|X_0| + \tilde q$  to obtain $|X_{k+1}| \le \bar q^{k+1}|X_0|$ for all $k\ge 0$. Therefore $X_k\rightarrow 0 $ as 
	$k\rightarrow \infty$ and it also follows that $g_k\rightarrow 0$ 
	in $ C^1( \overbar\VR^d,\VR^d)$. 
	\newline
	
(ii) Now we show that $(F_k)$ is a Cauchy sequence in $\Cf$. Recall that  by definition $F_m=(\Id+g_{m-1})\circ \cdots \circ (\Id+g_0)$ for all 
$m\ge 1$ and $F_0=\Id$. 
Hence using the triangle inequality and the right-invariance of $d(\cdot,\cdot)$ gives
\ben\label{eq:fm_fm_n}
\begin{split}
	d(F_m,F_{m+n+1}) & = d((\Id+g_{m-1})\circ \cdots \circ 
	(\Id+g_0),(\Id+g_{m+n})\circ \cdots \circ (\Id+g_{m-1})\circ \cdots \circ (\Id+g_0))\\
	& = d(\Id ,(\Id+g_{m+n})\circ \cdots \circ (\Id+g_m))  \le \sum_{\ell=m}^{n+m} d(\Id, \Id+g_\ell)
\end{split}
\een
for all $m,n\ge 1$. 
Further, in view of Lemma~\ref{lem:estimate_metric} and estimate 
\eqref{eq:convergence_rate}, we get for all $\ell\ge 0$, 
\ben\label{eq:est_gl_id}
\begin{split}
	d(\Id, \Id+g_\ell) & \le \|g_\ell\|_{C^1} + \|g_\ell\|_\infty + 2\|\partial g_\ell\|_\infty (\|\partial g_\ell\|_\infty + 1)\\
	& \le 2(1+(\alpha^\ell\|g_0\|_{C^1} + 1))\alpha^\ell \|g_0\|_{C^1} \le 5 \alpha^\ell \|g_0\|_{C^1}.
\end{split}
\een
So using the previous inequality together with $\alpha<1$ to further estimate \eqref{eq:fm_fm_n} we find
\ben\label{eq:apriori_}
d(F_m,F_{m+n+1}) \le 5\|g_0\|_{C^1} \sum_{\ell=m}^{n+m} \alpha^\ell = 5\|g_0\|_{C^1}\alpha^m\frac{(1-\alpha^{n+1})}{1-\alpha}.
\een
 The right hand side of 
\eqref{eq:apriori_} tends to zero  as $m,n\rightarrow 0$. This shows that $(F_m)$ is a Cauchy sequence in complete metric space
$\Cf$ and therefore we find 
$F_*\in \Cf$, such that $d(F_m, F_*) \rightarrow 0$ as $ m\rightarrow \infty.$ 
Hence passing to the limit $n\rightarrow \infty$ in \eqref{eq:apriori_} yields the a-priori estimate \eqref{eq:apriori}. It remains to show that $F_*$ is a root.  Let use define
\ben
\Fv^i_*(y):= \hat \phi(x^i_*-y)\nu_{\partial \Omega_*}(x^i_*),\quad x^i_*:= F_*(x_i), \quad \Omega_* := F_*(\omega_0).
\een
 Thanks to Lemma~\ref{lem:convergence_F_n} we know that $F_k\rightarrow F^*$ in $\Cf$ as $k\to \infty$ implies $ F_k\circ (F^*)^{-1}-\Id\rightarrow 0$ in $\ac C^1(\overbar\VR^d,\VR^d)$ as $k\to \infty$. As a result we infer from Corollary~\ref{corr:esitmate_kernel_parallel_transport},
\ben
\begin{split}
\|\Fv_k^i - \Fv_*^i\|_{C^1} &= \|\Fv_{\partial \Omega_k}^{x_k^i}-\Fv_{\partial \Omega_*}^{x_*^i}\|_{C^1} = \|\Fv_{F_k\circ F_*^{-1}(\partial \Omega_*)}^{F_k\circ F_*^{-1}(x^i_*)}-\Fv_{\partial \Omega_*}^{x_*^i}\|_{C^1} \le  c_2\|F_k\circ (F^*)^{-1}-\Id\|_{C^1} \rightarrow 0 \quad \text{ as } k\to \infty,
\end{split}
\een
for all $i=1,\ldots, n$. Now employing the continuity properties of the first and 
second derivative, and $g_k\rightarrow 0$ in $\ac C^1(\overbar\VR^d,\VR^d)$, we can pass to the limit in the Newton equation \eqref{eq:newton_equation}.
This shows that $DJ(F^*(\Omega))(\Fv^*_{\sigma^*,i})=0$ for $i=1,\ldots, n$. 
\end{proof}

\section{Numerical aspects and applications}
The goal of this section is to verify the convergence rates proved in Theorem~\ref{thm:main_Newton}. For this purpose we study a simple shape function for which the global solution and stationary points are  known. We compare the solutions obtained with the boundary and domain Hessian and examine the influence of the boundary discretisation on the convergence rates.

For every bounded and open set $\Omega\subset \VR^2$ define the shape function 
\ben\label{eq:functional_numerics}
J(\Omega):=\int_{\Omega} \fsf \; dx, 
\een
where $\fsf\in C^2(\VR^2)$ is a given function and specified for two different test cases below. A global minimiser of the above shape function is given by $\Omega^* = \{\fsf<0\}$, however, this shape function exhibit infinitely many stationary points depending on the nature of $\fsf$.  Although this example might seem trivial it already features many difficulties when 
we use the domain Hessian. 

\subsection{Discrete setting}
 Let $\omega_0$ be a bounded domain with $C^1$ boundary. Then 
we approximate $\omega_0$ by a domain  $\omega_0^h\subset \VR^2$ that has a polygonal boundary $\partial \omega_0^h$ with vertices $\mathcal Y_0 := \{y_1,\ldots, y_N\}$, $N\ge 1$. We set $h:=1/N$. The set $\mathcal Y_0$ is assumed to be ordered and 
contained in $\partial \omega_0$. In this sense the set $\mathcal Y_0$ is an approximation of $\partial \omega_0$. We then select a subset $\mathcal X_0 = \{x_1,\ldots, x_n\}$ of $\mathcal Y_0$, $n\le N$, where $n$ corresponds to the number of approximate basis functions.  All subsequently appearing integrals over $\omega_0^h$ are evaluated using second order Lagrangian finite elements. The domain $\omega_0^h$ is then updated by moving the points
$\mathcal Y_0$.   

Let us now describe how we approximate the normal vector field along $\partial \omega_0^h$. Take three consecutive points  $y_{i-1}$,$y_i$ and $y_{i+1}$  in $\mathcal Y_0$. The normal of the edge between $y_{i-1}$ and $y_i$ namely $e_i:=\{sy_{i-1}+(1+s)y_i:\;s\in [0,1] \}$ is defined by $\nu_i := Je_i/|Je_i|$, where $J$ is the counter clockwise 90 degree 2D 
rotation matrix. We then define the normal at vertex $y_i$ by 
$\nu_i := (e_i+e_{i+1})/|e_i+e_{i+1}|$. 

Let now $F_k \in \Cf$ be a sequence of transformations. We define $\mathcal Y_k := \{F_k(y_1),\ldots, F_k(y_N)\}$ and $\mathcal X_k := \{F_k(x_1),\ldots, F_k(x_n)\}$. We denote by $\nu_i^k$ the normals 
constructed above using the polygon $\mathcal Y_k$.

For our experiments we use $\phi(r):= (1-r)_+^4(4r+1)$ to construct our basis functions $\Fv^i_k := \Fv_{\mathcal Y_k}(y):=\phi(|F_k(x_i)-y|/\sigma_k)\nu_k^i$, $\sigma>0$. 
We update $\sigma$ in each iteration by $\sigma_k := \gamma \max_{i=1,\ldots, n-1}|F_k(x_i) - F_k(x_{i+1})|$, where $\gamma\ge 1$ is a factor determining how many basis functions fall into the influence cover of each basis function $\Fv^x$. 
As only small shape variations are considered, the number $n$ is kept constant. However, for large shape deformations one probably has 
to include new control points $x_i$ in order to keep the condition number of the Hessian 
within a computable range.

We now state the discrete analog of Algorithm~\ref{eq:algo_newton_discrete}.

\begin{algorithm}[H]
	
	
	
	\KwData{Let  $\gamma>0$ and $n, N\in \VN$ be given. Choose  
		$\Omega\subset \VR^d$ and 
		$\mathcal X_0 := \{x_1,\ldots, x_n\}\subset \Omega$. Let 
		$F_0:=\Id$. }
	
	
	initialization\;
	
	\While{ $k  \le  N $}{
		
		1.) Compute $X_k\in \VR^n$ as solution of
			$
			 H_k X_k = -l_k.
			$
			and set $g_k := P_k(X_k)$. 
		
		2.) Update $F_{k+1} \gets (\Id+g_k)\circ F_k$.

		3.)     Update $\mathcal X_{k+1} \gets \{F_{k+1}(x_1),\ldots, F_{k+1}(x_n)\}$.
		
		4.)       Update $\mathcal Y_{k+1} \gets \{F_{k+1}(y_1),\ldots, F_{k+1}(y_N)\}$. 
		
		5.)      Update $\Fv^{i}_{k+1} \gets  \Fv^{i}_k  $. 	
		
		\eIf{ If $|X_k|\le \gamma$: exit program }{            
			step accepted: continue program\; 
		}{
		no sufficient decrease: quit\;
	}            
	increase $k \gets k+1$\;
}
\caption{Newton algorithm} \label{eq:algo_newton_discrete} 
\end{algorithm}

\input{basis_function.tex}

\subsection{Newton methods}
\paragraph{Choice of Hessian}
At each iteration $\Omega_k$ we have two Hessians at our disposal, namely \eqref{eq:hess_bnd} and \eqref{eq:hess_volume},
\begin{align}
	H_{\Omega_k,J}^{\text{vol}}(X)(Y) & =  \int_{\Omega_k} T_1(X): \partial Y + T_0(X)\cdot Y \; dx,\\
	H_{\Omega_k,J}^{\text{bry}}(X)(Y)&= \int_{\partial \Omega_k}  (\nabla \fsf\cdot \nu_k + \kappa\fsf) (X\cdot \nu_k)(Y\cdot \nu_k)\; ds,
\end{align}
where $\nu_k$ and $\kappa_k$ denote the outward pointing unit normal vector field and the 
curvature of $\partial \Omega_k$, respectively. 
We know that both Hessians coincide when $X$ and $Y$ are restricted to normal fields along $\partial \Omega$. The Newton equation at iteration $k$ using $H_{\Omega_k,J}^{\text{vol}}$ reads: find $g_k^{\text{vol}}\in \text{span}\{\Fv^1_k,\ldots, \Fv^n_k\}$, such that
\ben\label{eq:hess_volume_numerics}
\int_{\Omega_k} T_1(g_k^{\text{vol}}): \partial Y + T_0(g_k^{\text{vol}})\cdot Y \; dx = \int_{\Omega_k} \VS_1: \partial Y + \VS_0\cdot Y \; dx\quad \text{ for all } Y \in \text{span}\{\Fv^1_k,\ldots, \Fv^n_k\},
\een
where we recall that $DJ(\Omega_k)(Y) = \int_{\Omega_k} \VS_1: \partial Y + \VS_0\cdot Y \; dx$ with $\VS_1 = f I$ and $\VS_0 = \nabla f$. The discrete volume shape Hessian and first derivative are given by  $H_k^{vol} := (H_{\Omega_k,J}^{\text{vol}}(\Fv^i_k)(\Fv^j_k))_{i,j=1,\ldots, n}$ and 
$(l_k^{bry})_{i=1,\ldots, n} = (\int_{\Omega_k} \VS_1: \partial \Fv^i_k + \VS_0\cdot \Fv^i_k \; dx)_{i=1,\ldots, n}$. 

The Newton equation at iteration $k$ using the boundary shape Hessian $H_{\Omega_k,J}^{bry}$ reads: find $g_k^{\text{bry}}\in \text{span}\{\Fv^1_k,\ldots, \Fv^n_k\}$, such that
\ben\label{eq:hess_boundary_numerics}
\int_{\partial\Omega_k} (\nabla f\cdot \nu_k + \kappa_k f) (g_k^{\text{bry}}\cdot\nu_k) (Y\cdot \nu_k) \; ds = \int_{\partial\Omega_k} (S_1\nu_k \cdot \nu_k) (Y\cdot \nu_k)\; ds\quad \text{ for all } Y \in \text{span}\{\Fv^1_k,\ldots, \Fv^n_k\}.
\een
As in a stationary point $\Omega$ we have $f=0$ on $\partial \Omega$, we (as in \cite{MR3201954}) neglect $\kappa_k f$ in our experiments. Accordingly we take as discrete boundary shape Hessian and first derivative 
\[
H_k^{bry} := \left(\int_{\partial\Omega_k} \nabla f\cdot \nu_k  (\Fv^i_k\cdot\nu_k) (\Fv^j_k\cdot \nu_k) \; dx \right)_{i,j=1,\ldots, n}, \quad l_k^{bry} = \left(\int_{\partial\Omega_k} \VS_1\nu_k \cdot \nu_k (\Fv^i_k\cdot \nu_k)\; ds\right)_{i=1,\ldots, n}.
\]
In the boundary and domain Hessian case we run Algorithm~\ref{eq:algo_newton_discrete} with the discrete Hessian $H_k$ and the first derivative $l_k$ given by $H^{vol}_k$, $l^{vol}_k$ and $H^{bry}_k$, $l^{bry}_k$, respectively. 

We now replace the approximate normal functions $\Fv^i_k$ by the normal basis functions $\Fw^i_k$ defined in \eqref{eq:normal_basis}. Then neglecting the term 
$\fsf \kappa_k$  \eqref{eq:hess_boundary_numerics} becomes: find $g^{\text{bry}}_k(x) = \sum_{i=1}^n \gamma^i_k\nu_k(x) \ksf(x^i_k,x) \in \text{span}\{\Fw^1_k,\ldots, \Fw^n_k\}$,	so that
	\ben
	\int_{\partial\Omega_k} \nabla f\cdot \nu_k (g^{\text{bry}}_k\cdot \nu_k) (Y\cdot \nu_k) \; ds = -\int_{\partial\Omega_k} (\VS_1\nu_k \cdot \nu_k) (Y\cdot \nu_k)\; dx\quad \text{ for all } Y \in \text{span}\{\Fw^1_k,\ldots, \Fw^n_k\}.
	\een
The last equation is completely equivalent to: find $\gamma_k \in \text{span}\{\ksf(x^1_k,\cdot),\ldots, \ksf(x^n_k,\cdot)\}$, so that
	\ben
	\int_{\partial\Omega_k} \nabla f\cdot \nu_k  \gamma_k \alpha \; dx = -\int_{\partial\Omega_k} (\VS_1\nu_k \cdot \nu_k) \alpha \; dx\quad \text{ for all } \alpha \in \text{span}\{\ksf(x^1_k,\cdot),\ldots, \ksf(x^n_k,\cdot)\}.
	\een	
The function $\gamma^k$ is an approximation of $\gamma_{\text{ana}}^k := -\fsf /(\nabla\fsf\cdot \nu_k)\in C^\infty(\partial \Omega_k)$ which is precisely the solution of the Newton equation 
$\text{Hess} J(\Omega_k)[\gamma_{\text{ana}}^k] = - \grad J(\Omega_k)$. Also here we omit $\kappa_k f$ in our computation. We call $\gamma_{\text{ana}}^k$ approximated Riemannian Hessian. In each iteration the domain is then moved via $(\Id + g^{\text{bry}}_k)(\partial \Omega_k)$. In our numerical experiments we $\kappa_k := |\hat \gamma_{k+1}|/|\hat \gamma_k|$ as a measure of the speed of convergence, where $\hat \gamma_k = (\gamma_k^1,\ldots, \gamma_k^n)^\top$ is the coefficient vector 
corresponding to the expansion of $\gamma^k$ in the basis $w^l_k$.

\paragraph{Example 1: an ellipse}
As in \cite{MR3201954} we consider $J$ given by \eqref{eq:functional_numerics} with
$
\fsf(x_1,x_2) := (\mu x_1^2+ x_2^2 - 1).
$
The corresponding minimisation problem to minimise $J$ over $\Omega$ has a unique solution, the domain $\Omega$ enclosed by the ellipse $\{(x_1,x_2):\;x_1^2+ x_2^2 = 1\}$.

As the convergence rates are only proved for initial shapes sufficiently close to the stationary point, we choose a circle centered at the origin with radius $r=0.9$. 
We select $\mu =2$ to compare our results with \cite{MR3201954}. In the top row of Figure~\ref{fig:ellipse_comparision_hessians} the convergence rates of Newton's method using different number of control points are shown. In Figure~\ref{fig:snapshots_comparision_hessians} we show several snapshots of the shape progress.  We see that all three Hessians yield similar results. 

In Figure~\ref{fig:ellipse_comparision_hessians_different_interface_numbers} we study the dependence of the convergence rates on the
number of boundary points. Notice that the number of boundary points is not equal to the number of approximate basis functions. In fact in 
Figure~\ref{fig:ellipse_comparision_hessians_different_interface_numbers}  the number of basis functions is kept constant at $N=60$ and 
the number of boundary points range from $N_{int}\approx 100$ to $\approx 600$. We see that for all three Hessians the convergence rates improve when 
we choose more boundary points. In Figure~\ref{table:errors}  the corresponding function values are displayed. After iteration four the 
cost function value for all three methods coincide up to the 
sixth decimal place.  We observe that the function value is different for all three methods which means that the three methods compute three different, though very close, stationary points of $J$. 

\begin{figure}
	\begin{center}
		\includegraphics[width=0.99\textwidth]{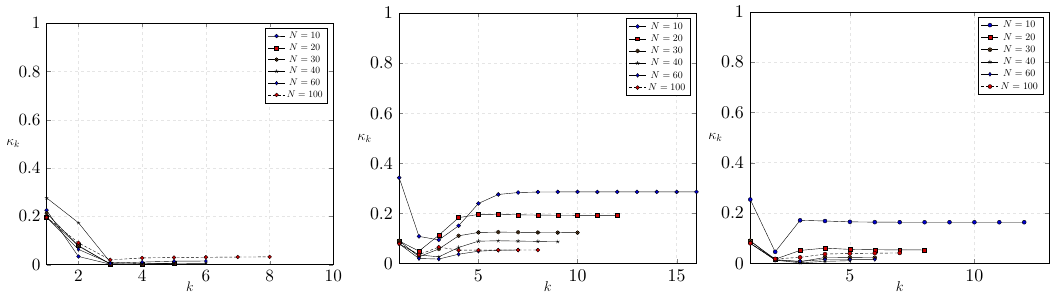}
		\includegraphics[width=0.99\textwidth]{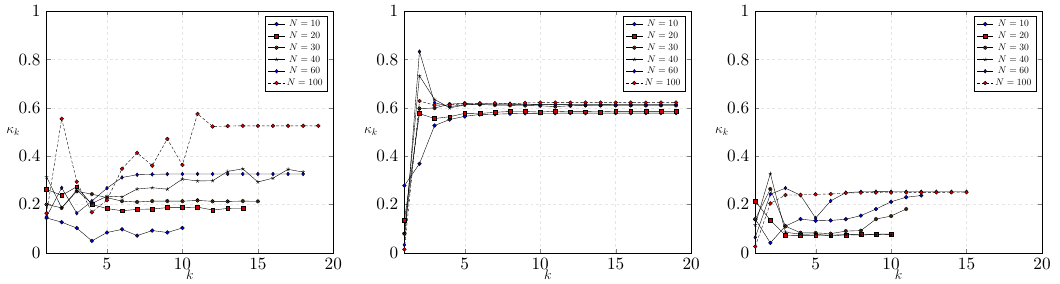}
		\caption{Top row: example 1 (ellipse); bottom row: example 2 (square); comparison of $\kappa_k := |X_{k+1}|/|X_k|$ using different number of control points; left: domain Newton method; middle: boundary Newton method; right: approximated Riemannian Hessian; algorithm is terminated when $|X_k|\le 1e-10$}
		\label{fig:ellipse_comparision_hessians}
	\end{center}
\end{figure}

\begin{figure}
	\begin{center}
		\includegraphics[width=0.9\textwidth,height=3.7cm]{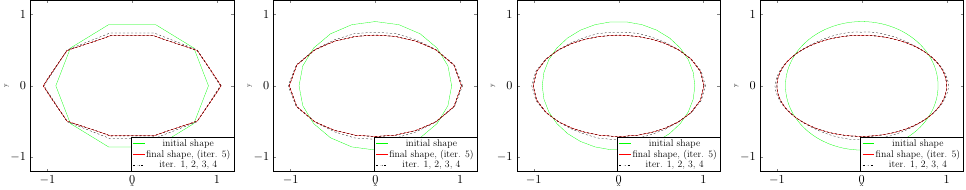}
		\includegraphics[width=0.9\textwidth,height=3.7cm]{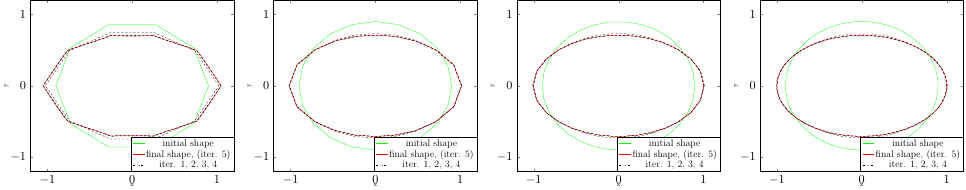}
		\includegraphics[width=0.9\textwidth,height=3.7cm]{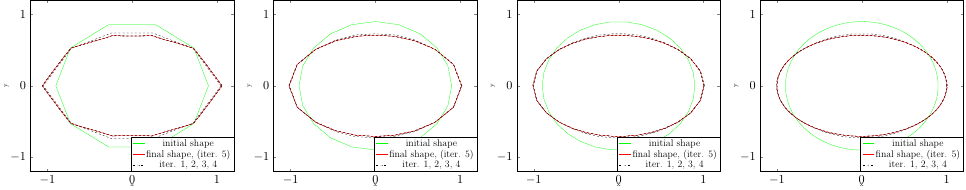}
		
		\caption{Shown are several snapshots of the shape progress for example 1 (ellipse); from left to right we used 
			$N=10,20,30$ and $100$ control points; from top to bottom: domain Hessian; boundary Hessian; approximated Riemannian Hessian}
		\label{fig:snapshots_comparision_hessians}
	\end{center}
\end{figure}

\begin{figure}
	\begin{center}
		\includegraphics[width=0.9\textwidth,height=3.7cm]{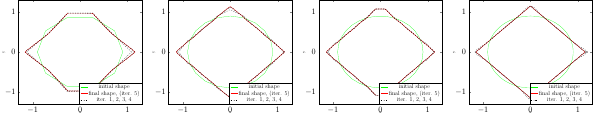}
		\includegraphics[width=0.9\textwidth,height=3.7cm]{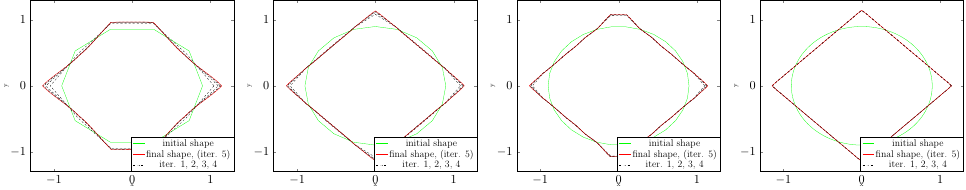}
		\includegraphics[width=0.9\textwidth,height=3.7cm]{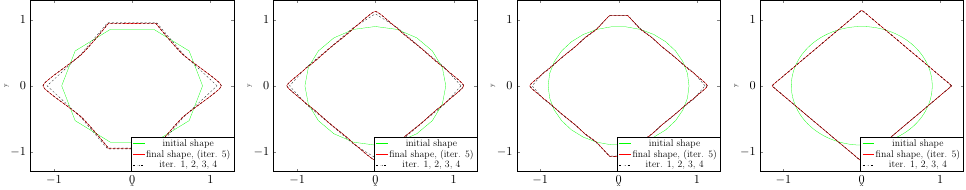}
		
		\caption{Shown are several snapshots of the shape progress for example 2 (square); from left to right we used 
			$N=10,20,30$ and $100$ control points; from top to bottom: domain Hessian; boundary Hessian; approximated Riemannian Hessian}
		\label{fig:snapshots_comparision_hessians_square}
	\end{center}
\end{figure}

\begin{figure}
	\begin{center}
		\includegraphics[width=0.99\textwidth]{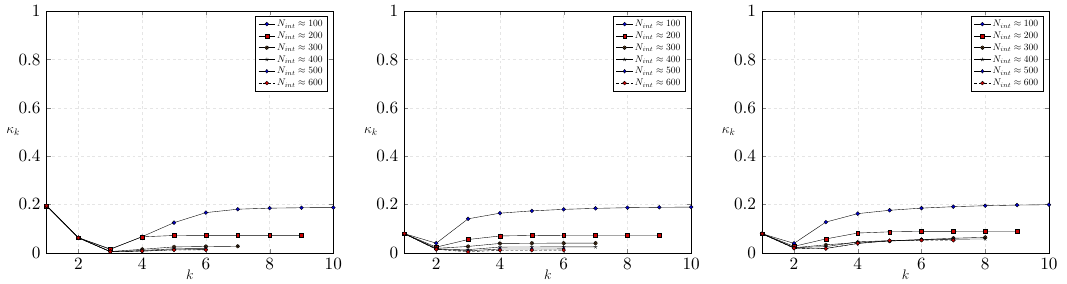}
		\caption{example 1 (ellipse); comparison of $\kappa_k := |X_{k+1}|/|X_k|$ for fixed number of control points $N=60$ and 
			different numbers of interface points $N_{int}$; left: domain Newton method; middle: boundary Newton method; right: approximated Riemannian Hessian; algorithm is terminated when $|X_k|\le 1e-10$}
		\label{fig:ellipse_comparision_hessians_different_interface_numbers}
	\end{center}
\end{figure}

\begin{table}
	\begin{center}
		\begin{tabular}{|c| c| c| c|c|c|c|} 
			\hline
			 iteration $k$ &   domain Hessian   &   boundary Hessian &     approx. Riemannian Hessian \\ [0.3ex] 
			\hline\hline
			    0    & -0.999584093291   			& -0.999584093291       & -0.999584093291	  \\ 
			    1    & -1.10590500673	  			& -1.1098402937 	  	& -1.10981595205    \\
				2	& -1.11070220243   		     	& -1.11072033698        & -1.11072034082	  \\
			 	3	& -1.11072032903   			    & -1.11072049441	    & -1.11072048816	  \\
			 	4	& -1.11072032951   			    & -1.11072049443	    & -1.11072048816	  \\
			 	5	& -1.11072032951   			    & -1.11072049443	    & -1.11072048816	  \\
			 	6	& -1.11072032951   			    & -1.11072049443 	    & -1.11072048816	  \\
			\hline
		\end{tabular}
		\caption{Example 1 (ellipse); $N=100$ control points column: $J(\Omega_k)$ using different Hessians; row: iteration; exact value of $\int_{\Omega} \fsf_1\; dx$ is -1.11072073431 when the ellipse $\partial \Omega$ is approximated by a polygon with 1400 points and the integral is evaluated using linear finite elements}
		\label{table:errors}
	\end{center}
\end{table}

\paragraph{Example 2: square}
As a second example we take $J$ as in \eqref{eq:functional_numerics} defined with the function
$
\fsf(x_1,x_2) :=  |x_1|+ |x_2| - 1.3.
$
Notice that $\fsf_1$ is weakly differentiable, but it is not continuous differentiable. The minimisation problem has a unique solutions, the domain $\Omega$ enclosed by the square $\{(x_1,x_2):\; |x_1|+ |x_2| = 1.3\}$. In order
to compute the second derivative of $\fsf$ we first  $L_2$ project the first derivative onto linear finite elements on $\Omega$ and take the second derivative of this derivative. The convergence results are shown in the bottom row of Figure~\ref{fig:ellipse_comparision_hessians}. The numerical algorithm is terminated if either $|X_k|\le 1e-10$ or if the maximal iteration number of 20 is reached. Some snapshots of the iterations are shown in Figure~\ref{fig:snapshots_comparision_hessians_square}.

\subsection{Gradient method}
\paragraph{Euclidean metric}
We now compare the difference between a gradient and Newton method. For this purpose we choose the Eulcidean metric (see \cite{eigelsturm16}) on the approximate space as inner product. The Euclidean metric is defined by
$
(\Fv^i_k,\Fv^j_k)_{\Omega_k} := \delta_{ij}
$
and extend this inner product to $\text{span}\{\Fv^1_k, \ldots, \Fv^n_k \}$. 
Then the steepest descent direction in this metric given as solution $g_k\in \text{span}\{\Fv^1_k, \ldots, \Fv^n_k \}$ of
\ben
(g_k,Y)_{\Omega_k} = - \int_{\Omega_k} \VS_1:\partial Y + \VS_0\cdot Y \; dx \quad \text{ for all } Y \in \text{span}\{\Fv^1_k, \ldots, \Fv^n_k \}.
\een
In each step the domain $\Omega_k$ is updated via $(\Id+s_k g_k)(\Omega_k)
$ with $s_k >0$ denoting the step size. It is readily seen (cf. \cite{eigelsturm16}) that
$
g_k = - \sum_{\ell=1}^n DJ(\Omega_k)(\Fv^\ell_k)\Fv^\ell_k.
$
As an initial shape we take again the domain $\Omega$ enclosed by the circle centered at the origin with radius $0.9$. A constant step size of $s_k = 0.4$ has been chosen. We terminate the algorithm if either $|X_k|\le 1e-10$ or after maximum of $100$ iterations. The results for the square and ellipse are depicted in Figure~\ref{fig:snapshots_comparision_gradient}. The difference between the Newton method is both visible from the shape progress and the convergence speed.

\begin{figure}
	\begin{center}
		\includegraphics[width=0.99\textwidth]{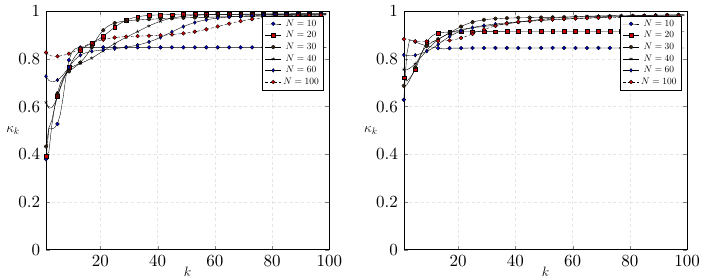}
		\caption{gradient algorithm with Euclidean metric; y-axis depicts $\kappa_k := |X_{k+1}|/|X_k|$ and x-axis depicts iteration number;  different number of control points $N$ are employed; left: example 1 (ellipse); right: example 2 (square) algorithm is terminated after a maximum of $100$ iterations}
		\label{fig:ellipse_comparision_gradient}
	\end{center}
\end{figure}

\begin{figure}
\begin{center}
	\includegraphics[width=0.9\textwidth,height=3.7cm]{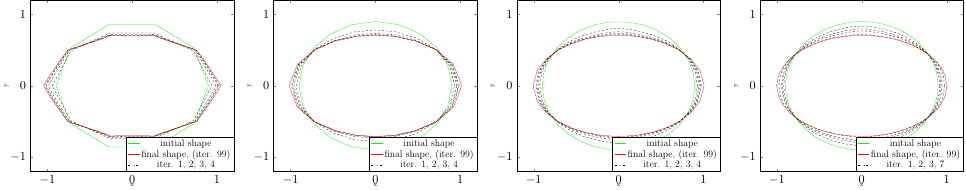}
	\includegraphics[width=0.9\textwidth,height=3.7cm]{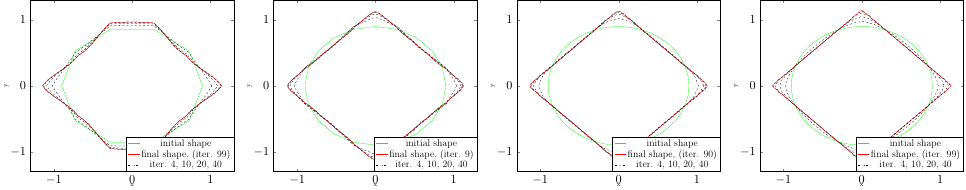}	
	\caption{Shown are several snapshots of the shape progress for using a gradient method with Euclidean metric; top row: example 1 (ellipse); bottom row: example 2 (square); from left to right we used 
		$N=10,20,30$ and $100$ control points}
	\label{fig:snapshots_comparision_gradient}
\end{center}
\end{figure}

\subsection*{Conclusion}
In this paper we have examined a Newton method defined via approximate normal functions that can be interpreted as the discretised version of an infinite dimensional 
Newton method. We introduced two different notions of Hessian, the domain and boundary Hessian.  
We proved superlinear convergence of a Newton method using the domain Hessian. 
In general quadratic convergence is lost when the vector fields are only approximated by 
approximate normal functions. Finally our results are validated by numerical experiments studying the convergence rates dependent on the discretisation.

The thorough numerical investigation using our approximate normal functions   also indicates that  for large shape deformations 
 the usage of the boundary shape Hessian
$H^{\text{bry}}_{\Omega,J}$ is favorable.  This can be explained by the fact that according to Theorem~\ref{thm:structure_second} the domain Hessian also contains tangential components which are not entirely eliminated by the 
approximate normal functions. However when we are close to a stationary domain no
significant difference has been observed. Nevertheless for some applications it may make sense to use the domain Hessian and therefore in order to allow for larger shape deformation as well different basis functions that are "more" normal to the boundary have to be found. Here the difficulty lies in the fact that the 
domain expression has to be evaluated with vector fields defined on $\VR^d$.  The search for such novel functions is challenging topic will be part of a future project.

\section*{Appendix}
\begin{proof}[Proof of Lemma~\ref{lem:contraction}]
We prove the lemma by induction. At $k=0$ we have by \eqref{eq:ak}
\ben
\begin{split}
	a_1 & \le \underbrace{c a_0}_{\le q_1} a_0 + q_2a_0 \le (q_1+q_2)a_0 = \bar q.  
\end{split}
\een	
Now suppose the results holds for all $k=0,\ldots, n-1$. Then from  \eqref{eq:ak}
\ben
\begin{split}
	a_{n+1} &  \le c\underbrace{a_n^2}_{\le \bar q^{2n} a_0^2}  + q_2 \bar q^{n} a_0  
	  \le \underbrace{ca_0}_{\le q_1} \bar q^{2n} a_0 + q_2 \bar q^{n} a_0
	 \le  \bar q^n (\underbrace{q_1 \bar q^n}_{\le q_1} + q_2  )a_0
	 \le (q_1+q_2)\bar q^{n} a_0 = \bar q^{n+1} a_0.
\end{split}
\een
This shows that \eqref{eq:ak} is also true for $k=n$ which finishes the proof. 	
\end{proof}

\begin{proof}[Proof of Lemma~\ref{lem:continuity_group}]
		Define $\Theta_k:= (\Id+f_k)\circ \cdots \circ (\Id+f_n)$, $k=1,\ldots, n$.  Then 
		it is readily checked that the recursive formula 	
		$
		\Theta_k = \Theta_{k+1} + f_k\circ \Theta_{k+1}
		$
		for $k=1,\ldots, n-1$ holds.  Summing over $k=1,\ldots, n-1$ and recalling 
		the telescope sum, we get
		\ben\label{eq:recursion}
		\sum_{k=1}^{n-1} f_k\circ \Theta_{k+1} = \sum_{k=1}^{n-1}(\Theta_k-\Theta_{k+1}) = \Theta_1 
		- \Theta_n. 
		\een
		Then \eqref{eq:recursion} together with the fact that $\Theta_k$ are 
		homeomophisms yield
		\ben\label{eq:est_1}
		\begin{split}
			\|\Theta_1-\Id\|_\infty &\le \|f_n\|_\infty +  \sum_{k=1}^{n-1}  \|f_k\circ 
			\Theta_{k+1}\|_\infty
			= \sum_{k=1}^n  \|f_k\|_\infty
		\end{split}
		\een
		Now observe that for all $k$,
		we have
		$
		\|\partial \Theta_k\|_\infty  \le (1+\|\partial f_k\|_\infty)\cdots(1+\|\partial f_n\|_\infty)
		\le e^{\sum_{\ell=k}^n\|\partial f_\ell\|_\infty}
		$ and consequently
		\ben\label{eq:est_2}
		\begin{split}
			\|\partial \Theta_1-I\|_\infty &\le \|\partial f_n\|_\infty +  
			\sum_{k=1}^{n-1}  \|(\partial f_k)\circ 
			\Theta_{k+1}(\partial \Theta_{k+1})\|_\infty \\
			&  \le  \|\partial f_n\|_\infty +  
			\sum_{k=1}^{n-1} e^{\sum_{\ell={k+1}}^n\|\partial f_\ell\|_\infty} \|\partial f_k\|_\infty\\
			& \le \|\partial f_n\|_\infty +  
			e^{\sum_{\ell=1}^n\|\partial f_\ell\|_\infty}\sum_{k=1}^{n-1} \|\partial 
			f_k\|_\infty\\
			& \le e^{\sum_{\ell=1}^n\|\partial f_\ell\|_\infty} \sum_{k=1}^n\|\partial 
			f_k\|_\infty,
		\end{split}
		\een
		where in the last step we used $e^\eps \ge 1$ for all $\eps \ge 0$. 
		Now \eqref{eq:est_1} and \eqref{eq:est_2} together yield \eqref{eq:estimate_f_k}.
\end{proof}

\bibliographystyle{plain}
\bibliography{refs}
\end{document}

%% file: slidesymbols.tex
\providecommand{\Div}{\operatorname{div}}          
\providecommand{\grad}{\operatorname{{\bf grad}}}       

\providecommand{\Supp}{\operatorname{supp}}                            
\providecommand{\supp}{\Supp}

\providecommand{\Id}{\Op{Id}}                     

\newcommand{\VN}{{\mathbf{N}}}

\newcommand{\VR}{{\mathbf{R}}}
\newcommand{\VS}{{\mathbf{S}}}

\newcommand{\VV}{{\mathbf{V}}}

\newcommand{\Dsf}{\mathsf{D}}

\newcommand{\Ksf}{\mathsf{K}}

\newcommand{\fsf}{\mathsf{f}}

\newcommand{\ksf}{\mathsf{k}}

\providecommand{\Cf}{{\cal F}}
\providecommand{\Cg}{{\cal G}}

\providecommand{\Fg}{\mathfrak{g}}

\providecommand{\Fl}{\mathfrak{l}}

\providecommand{\Fv}{\mathfrak{v}}
\providecommand{\Fw}{\mathfrak{w}}

\newcommand*{\Op}[1]{\mathsf{#1}} 



%% file: basis_function.tex
\begin{figure}
	\begin{center}
		\resizebox{0.4\textwidth}{!}{
	\begin{tikzpicture}

	\coordinate (a) at (0, 0){};
	\coordinate (b) at (1, 2.5){};

	\coordinate (c) at (3.5, 3.5){};
	\coordinate (d) at (5, 2.5){};
	\coordinate (e) at (6, 1.5){};
	\coordinate (f) at (5, 0){};
	\coordinate (g) at (3.5, -1.5){};
	\coordinate (h) at (1.5, -1.5){};
	\coordinate (i) at (1.5, -1.5){};
	\coordinate (j) at (1.5, -1.5){};
	\coordinate (k) at (1.5, -1.5){};

	\draw[point] let \p{a}=(a),\p{b}=(b) in ($  ((\p{a})!.50!(\p{b})$) node[point]{}; 
	\draw[point] let \p{b}=(b),\p{c}=(c) in ($  ((\p{b})!.50!(\p{c})$) node[point]{}; 
	\draw[point] let \p{c}=(c),\p{d}=(d) in ($  ((\p{c})!.50!(\p{d})$) node[point]{}; 
	\draw[point] let \p{d}=(d),\p{e}=(e) in ($  ((\p{d})!.50!(\p{e})$) node[point]{}; 
	\draw[point] let \p{e}=(e),\p{f}=(f) in ($  ((\p{e})!.50!(\p{f})$) node[point]{}; 
	\draw[point] let \p{f}=(f),\p{g}=(g) in ($  ((\p{f})!.50!(\p{g})$) node[point]{}; 
	\draw[point] let \p{g}=(g),\p{h}=(h) in ($  ((\p{g})!.50!(\p{h})$) node[point]{}; 
	\draw[point] let \p{h}=(h),\p{a}=(a) in ($  ((\p{h})!.50!(\p{a})$) node[point]{};

	\draw[point][thick] (a)  -- (b) node[point]{} -- (c) node[point]{} -- 
	(d) node[point]{} -- (e)node[point]{} -- (f)node[point]{}  -- (g)node[point]{} 
	-- (h)node[point]{} -- (i)node[point]{} -- (j)node[point]{} -- (k)node[point]{} --  (a) node[point]{} ;
	
	
	
	\node (ab1) at ($  (a)!.50!(b)$) {};
	\node (bc1) at ($  (b)!.50!(c)$) {};
	\node (cd1) at ($  (c)!.50!(d)$) {};

	\node (t1) at ($(b)- (a)$){};
	\node (t2) at ($(c)- (b)$){};
	\node (t3) at ($(d)- (c)$){};
	
	\node (n1) at ($  (0,0)!1cm!90:(t1)$) {};
	\node (n2) at ($  (0,0)!1cm!90:(t2)$) {};
	\node (n3) at ($  (0,0)!1cm!90:(t3)$) {};

	\draw[->,orange] ($(ab1)$) -- ( $(n1)+(ab1)$);
	\draw[->,orange] ($(bc1)$) -- ( $(n2)+(bc1)$);

	\node (n12) at ($  (0,0)!1cm!($(n1) + (n2)$)$) {};
	\node (n23) at ($  (0,0)!1cm!($(n2) + (n3)$)$) {};

	\draw[->,red] ($(b)$) -- ( $(n12)+(b)$);
	\draw[->,red] ($(c)$) -- ( $(n23)+(c)$);

	\node[anchor=north,text width=2cm] (solution) at (4,2) {
		$\Omega$};
	
	\node[text width=1cm] (solution) at ($(b)+(-1,0.2)$) {
		\begin{align*}
		\nu(x_0)
		\end{align*}};
	
	\node[text width=1cm] (solution) at ($(b)+(0.5,0.4)$) {
		$x_0$};
	
	\node[text width=1cm] (solution) at ($(c)+(1,0.2)$) {
		$x_1$};
	
	\node[text width=1cm] (solution) at ($(d)+(1,0.2)$) {
		$x_2$};
	
	\node[text width=1cm] (solution) at ($(e)+(1,0.2)$) {
		$x_3$};
	
	\node[text width=1cm] (solution) at ($(f)+(0.8,-0.2)$) {
		$x_4$};
	
	\node[text width=1cm] (solution) at ($(g)-(-0.3,0.3)$) {
		$x_5$};
	
	\node[text width=1cm] (solution) at ($(h)-(0.2,0.2)$) {
		$x_6$};
	
	\node[text width=1cm] (solution) at ($(a)-(0.2,0.2)$) {
		$x_7$};

	


	


	
	

	
	

	\end{tikzpicture}
		}
	\caption{Schematic polygonial domain $F_k(\omega_0^h)$ with vertices 
		$\{x_0,\ldots, x_7\}$ and normal field $\nu$. }
	\end{center}
\end{figure}